\documentclass[12pt]{amsart}
\usepackage{extsizes}

\usepackage{amssymb,amsfonts,amsthm,amsmath}
\usepackage{graphicx}
\usepackage{color}
\usepackage[mathscr]{eucal} 
\usepackage{subcaption}
\usepackage{mathtools}
\usepackage[pagebackref=true]{hyperref}
\usepackage{hyperref}
    \hypersetup{
          unicode   = true,
          colorlinks= true,
    }
\usepackage{cite}

\usepackage{hyperref}
\usepackage{cleveref}
\crefformat{section}{\S#2#1#3} 
\crefformat{subsection}{\S#2#1#3}
\crefformat{subsubsection}{\S#2#1#3}

\AtBeginDocument{
\addtolength\abovedisplayskip{-0.1\baselineskip}
 \addtolength\belowdisplayskip{-0.1\baselineskip}
 \addtolength\abovedisplayshortskip{-0.1\baselineskip}
  \addtolength\belowdisplayshortskip{-0.1\baselineskip}
}
\usepackage[left=1 in,top=1 in,right=1 in, bottom=1 in]{geometry}
\usepackage[T1]{fontenc} 
\numberwithin{equation}{section}
\definecolor{darkred}{rgb}{.70,.12,.20}

\definecolor{darkgreen}{rgb}{.20,.52,.14}

\definecolor{byz}{rgb}{.44,.16,.39}

\numberwithin{equation}{section}
\DeclareMathOperator{\supp}{supp}
\newtheorem{lemma}{Lemma}
\newtheorem{remark}{Remark}
\newtheorem{definition}{Definition}
\newtheorem{corollary}{Corollary}
\newtheorem{assumption}{Assumption}
\newtheorem{theorem}{Theorem}

\newtheorem{axiom}{Axiom}
\newtheorem{proposition}{Proposition}
\newtheorem*{Axiom}{Axiom of Mass Conservation}
\newtheorem{corollary }{Corollary}
\newtheorem{Lad-Ur}{Ladyzhenskaya-Uraltceva iterative Lemma}
\newtheorem{hypotheses}{Hypotheses}
\newtheorem{exmpl}{Example}
\usepackage{todonotes}

\usepackage{enumitem}

\title[Iterative Energy Estimate - Degenerate Einstein Brownian model]{An Iterative Energy Estimate for Degenerate Einstein model  of Brownian motion }
\author{
Isanka Garli Hevage$^1$, Akif Ibraguimov$^2$, Zeev Sobol$^{3}$
}
\date{}
\begin{document}
\maketitle
\begin{center}
{$^1$ Department of Mathematics and Statistics, Texas Tech University\\
Lubbock, Texas, USA, e-mail: \texttt{isankaupul.garlihevage@ttu.edu}
\smallskip
\\
{$^2$ Department of Mathematics and Statistics, Texas Tech University,}
\\
\small{Lubbock, Texas, USA, e-mail: \texttt{akif.ibraguimov@ttu.edu}}
\smallskip
\\
{$^3$ Department of Mathematics, University of Swansea,}
\\
{Fabian Way, Swansea SA1 8EN, UK, e-mail: \texttt{z.sobol@swansea.ac.uk}}
}
\end{center}
\begin{abstract}
\noindent
We consider the degenerate Einstein’s Brownian motion model when the time
interval $\tau$ of free jumps (particle-jumps before the collisions), reciprocals to the
number of particles per unit volume $u(x, t) \geq 0$, at the point of observation $x$
at time $t$. The parameter $\tau \in (0, C]$, which controls the characteristics of
the fluid,  ”almost decreases”, with respect to $u$, and converges to $\infty$  as $u \rightarrow 0$. This degeneration leads to the localization of the particle-distribution in the media.  In the paper, we present a structural condition of the time interval and the frequency of these \textit{free jump}s as  functions of $u$ which guarantees the finite speed of propagation of $u$. 
\end{abstract}
\section{Introduction}
Consider the thought experiment of the fluid which occupy a bounded domain $\Omega\subset \mathbb{R}^{N}$  by the particles for $ N=1,2,\dots$. Let $\vec{\Delta} \triangleq <\Delta_{1},\Delta_{2}, \cdots,\Delta_{N}>$ be the $N$ dimensional vector of \textit{\textit{free jump} }, which we  define  as the particles' jumps without collisions. In definition of \textit{\textit{free jump}} we follow the classical Einstein paradigm in his famous thesis \cite{Einstein56}, where in  literature the term "free pass" is used instead (see \cite{vin-krug}). From this point of view events of \textit{free jump} and "free pass" are synonyms. \\
Einstein assumed that process of random motion of the particles is characterized by events of \textit{free jumps}. He proposed two main parameters which characterize \textit{free jumps} : the time interval $\tau$ within which the \textit{free jump}s occur, and frequency $\varphi(\Delta)$  of the occurrences of \textit{free jump}s of the length $\Delta$. Einstein assumed that diffusion: $ \displaystyle {\left[\int_{-\infty}^{\infty} {\Delta}^2 \varphi(\Delta)\ d \Delta\right] \bigg/ \tau } $ is to be constants, which allowed him to reduce his thought experiment to the classical heat equation with constant coefficients, where that exhibits the so-called effect of infinite speed of the perturbation, namely, \textit{if $u(x_0,t_0)$ be the concentration of particles at some moment of time $t_0$ and at any point $x_0$, is positive  , then $u(x,t)>0$  for all $x$ and $t \geq t_{0}$}. This property of $u(x,t)$ is very unrealistic. The question which we asked in this article is as follows: Can Einstein's paradigm of \textit{free jump}s be generalised in such way that one sees finite speed of propagation of the particles, i.e., \textit{if $u(x,0)=0$ at some point $x_0$ then $u(x_0,t)=0$   during time interval $[0,T]$ for some $T>0$} ?\\ 
In this article, we will present that if time interval of \textit{free jump}s $\tau$ is inverse proportional to the concentration $u$, and/or $u$ is proportional to  the variance $\displaystyle \sigma^{2} \triangleq \left[\int_{-\infty}^{\infty} {\Delta}^2 \varphi(\Delta)\ d \Delta\right]$, then under some assumptions of the proportionality, our solution of degenerate Einstein equation in IBVP\eqref{ibvp} will exhibit the finite speed of propagation, which is closely relate to so-called Barenblatt solutions for degenerate porous medium equation (see \cite{Barenblatt-96}).
The origin of the porous medium equation is differ from thought experiment of the Einstein paradigm. Namely, the first continuity equation hypothesised in the form:
\begin{equation}\label{cont-flux-eq}
L(\rho,\vec{J} )\triangleq\rho_t+\nabla\cdot\left(\rho \vec{J}\right)=0,
\end{equation}
where  density $\rho \in \mathbb{R}$, flux $\vec{J} \in \mathbb{R}^N$ in porous media that equals to the velocity $\vec v$ of the flow  (see \cite{Aronson-1}). It is assumed that $\rho$ is function of the pressure $p$ and can be approximated based on thermodynamic experiment. For example, $\rho = p^{\lambda}$ for some gasses by thermodynamic laws and  $ \displaystyle \vec{v}=-\frac{k}{\mu}\nabla p$, where $k$ - permeability of porous media and $\mu$ is viscosity subject to the experimental Darcy equation (see \cite{Muskat}). Combining these experimental relations in \eqref{cont-flux-eq} one can get the following divergent equation for scalar pressure function:
\begin{align}\label{cont-p-eq}
 L (w)=
 (w)_t - \frac{k}{\mu} \ \frac{1}{\lambda+1} \ \Delta \left(w^{ \frac{\lambda+1}{\lambda}}\right) = 0,
\end{align}
where $w = p^\lambda$ and $ \lambda > 0 $. In the pioneering work by G. I. Barenblat,  A.S. Kompaneetz and Ya.B. Zeldovich (see \cite{Evan}), it was shown that under specific initial and boundary conditions there exists a self-similar solution of the equation \eqref{cont-p-eq}, which exhibits finite speed  of propagation. Evidently, the Einstein operator $L_{E}$ in \eqref{M-1} is in nondivergent form that is based on the thought experiment, which allows to interpret the results on observable data in term of the length of \textit{free jump}s, and then adjust parameters $\tau$ and $\varphi$ of the model to execute further analysis.
In this paper, we will follows Einstein approach but will use technique for divergent equation.\\
The article is organized as follows. In \cref{Gen-Ein-Para}, we consider generalization of classical Einstein model  of Brownian motion to $\mathbb{R}^{N}$,  when the major parameters of the system, i.e., time interval $\tau$ of the \textit{free jumps} and the frequency $\varphi$ of the \textit{free jumps} depend on the concentration of the number of particle per unite volume $u$. We use the generic mass conservation law with absorption-reaction term  and basic stochastic principles to derive a partial differential inequality (PDI) under the Einstein's axioms in  \cref{Derivation PDI}. By introducing local forces and nonlocality processes, we conclude the deterministic PDI  \ref{M-1} which governs the dynamics of the generalized Einstein paradigm. Using Hypothesis \ref{tau} in  \cref{Nonlin-IBVP}, we define $u$ as the weakly approximated solution  to the nonlinear initial boundary value problem \ref{ibvp} such that holds the limit in \eqref{eps_sol}. We Introduce the assumption on functions $H,F$ and $G$ in  \cref{Assumptions} to explicitly structures coefficients as in the functions $H$ in \eqref{H-choice} and $F$ in \eqref{F-result} in a way that $u$ conserves the finite speed of propagation in  \cref{localization}. In \cref{iterative ineq}, we establish collaborating  Lemmas and introduce Ladyzhenskaya Iterative scheme which apply in  \cref{localization} to prove the localization property of $u$ based on De-Georgi's construction, if the initial energy functional $Y_{0}[T]$ preserves certain boundednes with respect to some $T' > 0$.  \cref{Exmp-P-F} is focused on models of the functions that attest all constrains on the functions $ F, G$ and $H$, which consequently guarantees the localization theorem \ref{localization}. We present auxiliary results of functional spaces in  \cref{WAPS} to prove the uniform boundednesses regarding the classical solution $u^{\epsilon}$  which ensures its limit as in \eqref{eps_sol}.

 
\section {Generalized Einstein paradigm}\label{Gen-Ein-Para} 
Let $u(x, t)$ be the 
function which represent  the number of particles per unit volume, suspended in the medium of interest at point $x \in \mathbb{R}^N$ and at time $t > 0 $ . Denote $\mathbb{J}(\tau)$ to be the set of vectors with noncolliding jumps corresponding to the time interval $\tau$. We call  $\vec{\Delta} \in \mathbb{R}^N $ to be a \textit{vector of free jump of particles} if $\vec{\Delta} \in  \mathbb{J}(\tau)$.
Assume the following extension of the axioms in classical Einstein Brownian motion.
\begin{assumption}   \normalfont
\begin{enumerate} 
\item
\label{constitive} Time interval of \textit{free jump}s $\tau$, expected vector $\vec{\Delta}_e$  of  a \textit{free jump} $\vec{\Delta}$ and probability density function of \textit{free jump} $\varphi(\vec{\Delta})$  are the only parameters  which characterize process of \textit{free jump}s. Note that in a view of the definition of the set $\mathbb{J}(\tau)$, if $\vec{\Delta} \notin \mathbb{J}(\tau)$ then  $\varphi(\vec{\Delta})=0$.
\item
 The key parameters $\tau$ and $\varphi$ can depend on the concentration of the particles and also the space-time coordinate.
\end{enumerate}
\end{assumption}
\begin{axiom}  \normalfont {Whole universe axiom:}
 \begin{align}\label{uni-ax}
      \int_{\mathbb{J}(\tau)}\varphi(\vec{\Delta})d\vec{\Delta} = 1.
 \end{align}
\end{axiom}

\begin{assumption}  \normalfont {Symmetry of \textit{free jump}s}:
\begin{equation}\label{symmetry}
\varphi(\vec{\Delta}) = \varphi(-\vec{\Delta}).    
\end{equation}
Existence of second moments:
\begin{equation}
    \int_{\mathbb{P(\tau)}}|\vec{\Delta}|^2\varphi(\vec{\Delta})d\vec{\Delta}<\infty.
\end{equation}
\end{assumption}

Note that it follows from the preceding assumption that the expectation of $\vec{\Delta}$ equals zero, and there exists a covariance matrix $[\sigma_{ij}^{2}]$ of \textit{\textit{free jump}s}, where
\begin{align}
\sigma_{ij}^2 \triangleq  \int_{\mathbb{J}(\tau)}{\Delta_{i}}{\Delta_{j}} \varphi(\vec{\Delta})d\vec{\Delta}, \quad  \text{for }  i,j = 1,2,\dots, N.
\label{var-ij}
\end{align}
Evidently  $\sigma_{ij}^2(x,t)$ depend on space $x$ and time $t$. We postulate generalized Einstein's axiom for the number of particles found  at point $x$ at time $t+\tau$ a, in the control volume $dv$ by:
\begin{Axiom}\begin{equation}
u(x, t+\tau) \cdot  dv =  
\int_{\mathbb{J}(\tau)} u(x+ \vec{\Delta}, t)\varphi(\vec{\Delta}) d \vec{\Delta} \cdot dv
+
\tau\cdot\int_{\mathbb{J}(\tau)} M(x+ \vec{\Delta}, t)\varphi(\vec{\Delta}) d \vec{\Delta} \cdot
dv.\label{Einstein_conserv_eq}
\end{equation}
\end{Axiom}

The axiom intuitively  expressed that at any given point in space $x$ at time $t+\tau$, we will observe density of all particles with \textit{free jump}s from the point $x$ at time $t$,  "$+$ density" of  particles which "produced" and "$-$ density" of  particles which "consumed" during time interval $[t,t+\tau].$  For comparison, see the first formula with integral on page 14  in \cite{Einstein56} . We modeled \eqref{Einstein_conserv_eq} by adding the term $ M(\cdot)$ which reports the process of  rate of absorption-consumption during the time interval $[t,t+\tau]$ due to particles interaction  along of \textit{free jumps}. We  consider the scenario, when the consumption dominates over the production by setting $ M(\cdot) \leq 0$.
\begin{remark}
Einstein definition of density of particles in {\normalfont\cite{Einstein56}}  differs from the fundamental definition of the density of fluid, rather means the concentration of the volume of interest.
\end{remark}
\section{Derivation of Partial differential Inequality} \label{Derivation PDI}
In this section, we derive partial differential inequality whose solution exhibits the feature of finite speed of propagation.  
Let $\zeta = (\zeta_1,\zeta_2,\cdots,\zeta_N) \in$ be a multi-index and $ x^{\zeta} \triangleq x_{1}^{\zeta_1} \cdot x_{2}^{\zeta_2} \cdot \dots \cdot x_{N}^{\zeta_N} $ for $x, \zeta \in \mathbb{R}^{N}$. Assume that $u(x,t)\in C_{x,t}^{2,1} $. By Taylor's Expansion  \cite{kon} and using \eqref{uni-ax} - \eqref{var-ij}, we get 
\begin{align}
\int_{\mathbb{J}(\tau)}
u(x +\vec{\Delta},t)  \varphi(\vec{\Delta})d\vec{\Delta}
 \text{ = } 
\text{ } u(x,t) 
 \text{ + }
\frac{1}{2}\sum_{i,j=1} \sigma_{ij}^{2}u_{x_{i}x_{i}}{(x,t)}
\text{ + }
R_{\zeta}\label{Taylor-eq},
\end{align} 
where 
\begin{align}
R_{\zeta} \triangleq {\int_{\mathbb{J}(\tau) }
 \sum_{|\zeta| = 2},
H_{\zeta}(x,\vec{\Delta},t)(\vec{\Delta})^{\zeta}
\varphi(\vec{\Delta})d\vec{\Delta}} \label{R-term}
\end{align}
with locally bounded function $H_{\zeta}$ such that
$\lim_{\vec{\Delta}\rightarrow 0} H_{\zeta}(x ,\vec{\Delta},t) = 0.
$ 
Using \eqref{Taylor-eq} in $\eqref{Einstein_conserv_eq}$ yields
\begin{equation}
u(x , t + \tau) - u(x,t) = 
 \frac{1}{2}\sum_{i,j=1}^N  \sigma_{ij}{u}_{x_{i} x_{j}} {(x,t)} + R_{\zeta} 
 + 
\tau\int_{\mathbb{J}(\tau)} M(x+ \vec{\Delta}, t)\varphi(\vec{\Delta}) d \vec{\Delta} .
 \label{post-Taylor}
\end{equation}
Observe that \eqref{post-Taylor} is defined at different points in space and time . We eliminate this ambiguity and derive the equation at the same point, by using Carathéodory's differential criterion:     there exists a function $\psi^{t}$ such that 
\begin{align}\label{Caratheodory-1}\small
& u(x, t+ \tau)  - u (x, t) =
\tau \psi^{t} (x,t,\tau), 
\end{align}
where $  \lim_ {s \rightarrow 0} \psi^{t}(x,t,s) =u_{t}(x,t)$. Using  \eqref{Caratheodory-1} in \eqref{post-Taylor} we get 
\begin{equation}
\tau u_{t}(x,t) =
 \frac{1}{2}\sum_{i,j=1}^N  \sigma_{ij}^2{u}_{x_{i} x_{j}} {(x,t)} + R_{\zeta}
 + 
\tau\int_{\mathbb{J}(\tau)} M(x+ \vec{\Delta}, t)\varphi(\vec{\Delta}) d \vec{\Delta} .
 \label{post-Taylor-2}
\end{equation}
For the forthcoming iterative procedure, we assume that
\begin{assumption}
 \begin{align}
 R_{\zeta}+ \tau \int_{\mathbb{J}(\tau)} M(x+ \vec{\Delta}, t)\varphi(\vec{\Delta}) d \vec{\Delta}  \leq 0 . \label{Negative-pheno}
 \end{align}
\end{assumption}
{First term $R_{\zeta}$  is responsible for nonlocality of the process . $R^1$ reflects concentration jumps in time, for instance, due to birth and death in biological system.  The last term in \eqref{Negative-pheno} can be interpreted as a local force. For example if $M$  noise during time of observation and therefore it can be stochastic.}
Using   \eqref{Negative-pheno} in \eqref{post-Taylor-2} and approximating the first order terms, it follows that the function $u(x,t)$ can be estimated by the following partial differential inequality:
\begin{align}
   L_{E} u =
   \tau{u_t}-
    \sum_{i,j=1}^N &  \sigma_{ij}^2 u_{x_{i}x_{j}} \leq 0.
\label{M-1} 
\end{align}
Observe that the partial differential inequality \eqref{M-1} is deterministic, while the stochastic nature is modeled by functions  $\tau$ and $\varphi$ (the latter appears in \eqref{M-1} in the form of the covariance $\sigma_{ij}^2 $ in \eqref{var-ij}), which  are key characteristics of the process dynamics. In general, they can be functions of spatial $x$ and time  $t$ variables, concentration $u$ and its gradient $\nabla u$. In this report, we present $\tau$ and $\sigma_{ij}^2$ satisfy the following Hypothesis in addition to \eqref{uni-ax} and \eqref{symmetry}. 
\begin{hypotheses}\label{tau} 
Let $ P\in C[0,\infty)$ be a function such that 
\begin{equation}\label{def-P-compos}
    c_1P(u)|\xi|^2 \leq \sum\limits_{ij=1}^N \frac1\tau \sigma_{ij}^2\xi_i\xi_j\leq c_2P(u)|\xi|^2, \  \xi\in\mathbb{R}^N,
\end{equation}
for some $0<c_1<c_2$. Here $P(0) = 0$ and $ 0 < P(s) < c_3< \infty$ when $s>0$.
\end{hypotheses}
\begin{remark}
In \eqref{def-P-compos} $P(u)$, is a composite parameter of stochastic processes of \textit{free jumps}, which characterizes the relative covariance  with respect to the time interval of free jumps. In  our scenario, $P(u)$ degenerate at $u=0$ which mean that the growth due to dispersion is slower than the time interval of dispersion, as the concentration vanishes.  
\end{remark}
\section{Nonlinear Degenerate IBVP} \label{Nonlin-IBVP}
For the sake of simplicity, we analyse the case when the covariance matrix is diagonal and positively defined. For $i,j=1,\dots,N$, let  $\sigma_{ij}^2 = [\sigma^2(u)] \delta_{ij}$ for some $ [\sigma^2(u)] > 0$. Consequently $P=[\sigma^2(u)] \big/ [\tau(u)]$, and \eqref{M-1} takes form
\begin{align}
    u_t \leq P(u) \Delta u. \label{M-2}
\end{align}
In order to derive the nonlinear degenerate IBVP, and to study structural forms of the coefficients, we  define the functions $H$, $F$ and $G$ as follows.
\begin{definition}  \normalfont \label{G-def}
\begin{enumerate} [label={\normalfont (D-{\arabic*}) }]
\item
     Let $h > 0$ such that $h\in C(0,\infty)$  and integrable at 0. Then 
     \begin{align}
          H(u) \triangleq\int_0^u h(s) \ ds.   \label{H-F-def}
     \end{align}
    \item   Let $F \triangleq hP$ and $h$, $P$ be such that $F(0)=0$, and $F$ is differentiable on $(0,\infty)$ with a locally bounded derivative $F^{'}>0$. \label{F-def}
    \item   Let $G$ be such that  $\sqrt{F^{'}(u)} \triangleq G^{'}(u)$ and $G(0)=0$. Then $\displaystyle G(s)\triangleq \int_{0}^{s} \sqrt{F^{'}(s)} \ ds. $  \label{G-def}
\end{enumerate}
\end{definition}
\begin{remark}\label{G-F-2} 
Note that 
$ \displaystyle
0\leq G(u)=\int_0^u G^{'}(s)ds \leq \sqrt{u\int_0^u F^{'}(s)ds} = \sqrt{uF(u)}$, and all functions $F,G$ and $H$ are increasing on closed interval. 
\end{remark}
We multiply \eqref{M-2} by $h(u)$ and obtain
\begin{align}
 Lu\triangleq [ H(u)]_t -   F(u)\Delta u \quad \mbox{on } \Omega\times(0,T] , \label{L-OP}
\end{align}
for some $T>0$. Here $u$ is positive measurable function such that $u(\cdot,t) \to u(\cdot,0)$ as $t\to 0$ in local measure, with
$ \displaystyle
u\in L_{loc}^{\infty}\left(\Omega\times [0,T]\right)$, $ 
\nabla u\in L^{2}_{loc}\left( \Omega \times (0,T]\right)$ and $ u_t\in  L^{1}_{loc}\left( \Omega \times (0,T]\right).$

Hence $F(u)\Delta u$ is understood in the weak sense, i.e., for all
$\theta\in Lip_{c}(\Omega)$ we write
\begin{multline}\label{weak Laplace}
-\int_{\Omega} \theta F(u)\Delta u dx  = \int_{\Omega} \left[F(u)(\nabla u)(\nabla \theta) + F'(u) |\nabla u|^2\theta\right]dx\\
 = \int_{\Omega} \left[F(u)(\nabla u)(\nabla \theta) +  \left|\nabla G(u)\right|^2\theta\right]dx.
\end{multline}
Then under the Hypotheses \ref{tau}, we defined $u(x,t)$ as a nonnegative  solution of the following  partial differential inequality
\begin{align}
\text{IBVP } = \begin{cases} 
\  [H(u)]_t -   F(u)\Delta u   \ \leq 0 \  & \text{ in }  \Omega\times (0,T],
\label{ibvp} \\
\hspace{2.3 cm} u(x,0)   \  = 0    &\mbox{ in } \Omega^{'} \Subset \Omega,   \\
\hspace{2.3 cm} \ u(x,t)  =\hspace{0.08 cm}0    & \mbox{ on } \ {\partial \Omega \times (0,T]},
 \end{cases} 
 \end{align}
where $u(x,0)\geq 0$, is continuous on $\Omega$. The condition on the boundary $\partial\Omega\times(0,T]$ is not essential to prove the finite speed of propagation, namely, for every ball $B_{R}(x_{0}) \Subset \Omega^{'}$ and $c<1$, there exists $T'=T'(x_0,c)\in (0,T]$ such that $u(x,t)=0$ for all $(x,t) \in B_{cR}(x_0)\times[0,T']$ in Theorem \ref{Main-T}. Observe that $u$ in IBVP \eqref{ibvp} is degenerates when $u \rightarrow 0$. Therefore its solution $u(x,t) \not \in {C}^{2,1}_{x,t}(\Omega\times (0,T])$. 
Our result is qualitative and does not address the existence of the solutions, but the obtained property of the solution is applicable for a weakly approximated solution $u$ (see \cite{liberman}, \cite{Krylov}),  which  is defined  in the following regularized problem for some $u^{\epsilon} \in \mathcal{C}^{2,1}_{x,t}(\Omega\times (0,T]) \cap C_{0}(\Bar{\Omega} \times (0,T])$  as follows.
\begin{align} 
\text{IBVP$_{\epsilon}$ } =\begin{cases} 
 [H(u^{\epsilon})]_t -   (F(u^{\epsilon})+\epsilon)\Delta u^{\epsilon}   \ = 0 \  & \text{ in }  \Omega\times (0,T],
\label{ibvp-ep} \\
\hspace{3.4 cm} u^{\epsilon}(x,0)   \  = \epsilon    &\mbox{ in } \Omega^{'} \Subset \Omega,   \\
\hspace{3.5 cm} \ u^{\epsilon}(x,t)  =\hspace{0.08 cm} {\epsilon} \psi(x)        & \mbox{ on } \ {\partial \Omega \times (0,T]},
 \end{cases} 
 \end{align}
where $u^{\epsilon}(x,0)$ is continuous in $\Omega$. In this article, the constants in all estimates for $u^{\epsilon}$  do not depend on $\epsilon$. This observation allow us to pass to the limit in the final estimates, and conclude the localization property for the limiting function
\begin{equation}\label{eps_sol}
u(x,t)=\lim_{\epsilon\to 0}u^{\epsilon}(x,t),
\end{equation}   
which is considered as a weak passage to the limit (see \cite{CIL92}). The obtained function $u(x,t)$ is called a weakly approximated solution of the IBVP \eqref{ibvp}, when the first differential inequality is replaced by the differential equation $Lu=0$ , which will exhibit localisation property. Further details on weakly approximated solution $u$ of IBVP \eqref{ibvp}, will be presented in \cref{WAPS}. For in detail discussions on the existence of weakly approximated solutions, see \cite{liberman} and \cite{Krylov}.
Hence further we will assume that $\Omega$ is a  domain with Lipchits boundary.

\section{Assumptions on Functions and interpretation}\label{Assumptions}
Let us state main properties and additional assumption on the functions $H$ and $F$ on some finite domain: $ [0,M]$.
\begin{assumption}  \   \normalfont \label{Assumps}
\begin{enumerate} [label={\normalfont (A-{\arabic*})}]
    \item  $ 
   \text{Exists} \ C_{1} > 0 \mbox{ such that }  F(s)\leq C_{1}G^{'}(s)G(s)$, where $ s\in[0,M].$ 
   \label{A-1}
   \item  $ 
 \displaystyle  \text{Exists} \  C_{2}>0  \mbox{ such that }\left(\sqrt{sF(s)}\right)^{\lambda}\leq {C_{2}H(s)}$, where  $0 < \lambda< 2$  and $  s\in[0,M]. \label{A-2} $
\end{enumerate}
\end{assumption}
We will choose $H$ and $P$ such that 
Assumption \ref{Assumps} holds.
\begin{proposition} \label{P-1}
Let $\Lambda + 1= \dfrac{2}{\lambda}$. Assume $F$ and $H$ be as in \ref{A-2}. Then
\begin{align}\label{H-prop}
    {H(s)\geq
\left( H^{-\Lambda}(M) + \Lambda C_{2}^{1+\Lambda}\int_{s}^{M} \frac{1}{\tau P(\tau)} d\tau \right)^{-\frac{1}{\Lambda}}}.
\end{align}
\end{proposition}
\begin{proof}
By raising the estimate in remark \ref{G-F-2} to the power $\lambda$ and using \ref{A-2}, we get 
\begin{align}
  G^{\lambda}(s)\leq \left(\sqrt{sF(s)}\right)^{\lambda}  \leq {C_{2}H(s)}\label{before-beta}\\
 F(s) \leq C_{2}^{\frac{2}{\lambda}} \dfrac{H^{\frac{2}{\lambda}}(s)}{s}.
\end{align}
By \ref{F-def}, we write \eqref{before-beta} becomes
\begin{align}
\frac{h(s)}{H^{\Lambda+1}(s)} 
& \leq \frac{C_{2}^{1+\Lambda}}{sP(s)}  \label{before-int}
\end{align}
Then by integrating  \eqref{before-int} over $(s,M]$,  we obtain the estimate \eqref{H-prop}.
\end{proof}
The preceding proposition implies the following choice of function $H$.
\begin{definition}  \normalfont \label{PIH}
Let $P$ in Hypotheses \ref{tau} be such that \ 
$ \displaystyle
0 < \int_M^\infty\frac{ds}{sP(s)} <\infty
$, for some finite $M$.
Define
\begin{align}
     I(s) & \triangleq \int_{s}^{\infty}\frac{d\sigma}{\sigma P(\sigma)} \ ; ~ M > s > 0.\label{I-s}
\end{align}
Consequently
\begin{align}\label{H-choice}
 H(s) \triangleq 
    \left[\Lambda I(s)\right]^{-\frac{1}{\Lambda}} =
    \left(\Lambda \int_{s}^{\infty} \frac{1}{\tau P(\tau)} d\tau\right)^{-\frac{1}{\Lambda}},  
    ~ s>0.
\end{align}
\end{definition}
Function $P(s)$, and consequently $H(s)$ are defined only on bounded interval. In order notation to be simple, we extended $I(s)$ on whole axis in such order that $I(s)\geq c_4>0$ on $[M,\infty)$.
\begin{remark} \label{new-def}
\normalfont
$H$ in \eqref{H-choice} has $H(0) = 0$ since $P(0) = 0$. By  substituting \eqref{H-choice} in \eqref{H-F-def}, 
\begin{align}
    h(s) = \frac{1}{sP(s)}\left[  \Lambda I(s) \right]^{-\frac{1}{\Lambda}-1}
    = \frac{1}{sP(s)} H^{(\Lambda+1)}(s) 
    \ , ~ s>0. \label{h-ret}
\end{align}
Using \eqref{h-ret} in \ref{F-def}, we get
\begin{align}
    F(s)=h(s)P(s)=\frac{1}{s} H^{(\Lambda+1)}(s)=\left(\Lambda s^{\frac{\Lambda}{1+\Lambda}}I(s) \right)^{-\frac{1}{\Lambda}-1}.
    \label{F-result}
\end{align}
Consequently,
$
\left[sF(s)\right]^{\frac\lambda2} = 
H^{(\Lambda+1)\frac\lambda2}(s)=H(s), \label{equiv-1}
$
since $(\Lambda+1)\frac\lambda2=1$. Thus, \ref{A-2} holds with $C_2=1$.
\end{remark}
The next proposition imposes conditions on $P$ in such a way that $F$ 
is nonnegatively increasing on $(0,M]$.
\begin{proposition}\label{P-2}
Let $F=\left[\Lambda s^{\frac{\Lambda}{1+\Lambda}}I(s) \right]^{-\frac{1}{\Lambda}-1}$. Assume $\exists$ constants $A, B$ such that
\begin{align}
\sup_{0 < s < M} P(s)I(s)=  A, \label{A}\\
\limsup_{s\rightarrow 0} P(s)I(s) =  B < A. \label{a}
\end{align}
If  \ $\frac{\Lambda+1}{\Lambda} > A $ then $F'(s) > 0 $ on $(0,M]$,
and 
if \  $\frac{\Lambda+1}{\Lambda} > a$ then $\lim_{s \rightarrow 0}F(s) = 0$. 
\end{proposition}
\begin{proof}
First we prove that function $F$ is increasing. Observe that it suffices to prove that
\begin{equation}s \mapsto \displaystyle{ s^{\frac{\Lambda}{1+\Lambda}}I(s)},
\end{equation}
decreases on $(0,M]$.
Note that
\begin{align}
\displaystyle
\frac{d}{ds}\left( s^{\frac{\Lambda}{1+\Lambda}}I(s)\right)
=
\left( \frac{\Lambda}{\Lambda+1} \frac{1}{P(s)}\right)\left( P(s)I(s) -\frac{\Lambda+1}{\Lambda}\right)s^{-\frac{1}{\Lambda+1}}. \label{comp-deri}
\end{align}
Thus, together with   $[\Lambda +1] / [\Lambda] > A$, one has $F^{'}(s) > 0 $ for $s\in (0,M] $. Next, we show that
$ \displaystyle
 s^{\frac{\Lambda}{1+\Lambda}}I(s) \rightarrow \infty \text{ when } s\rightarrow 0$, which implies $\displaystyle \lim_{s \rightarrow 0}F(s) = 0$.
By \eqref{a}, for every $\epsilon > 0 $ there exists $s_{\epsilon} \in (0,M]$ such that $ I(s)P(s) < a+\epsilon$ for $s\in(0,s_\epsilon)$, which yields the following inequalities.
\begin{align*}
 I(s) < (a + \epsilon)\frac{1}{P(s)}  &  = -(a+\epsilon)s I^{'}(s),\\
    \frac{d}{ds} \ln I(s)  & < -\frac{1}{a+\epsilon} \cdot \frac{1}{s},\\
I(s) & >  I(s_{\epsilon})\left(\frac{s_{\epsilon}}{s}\right)^{\frac{1}{a+\epsilon}} \text{ for }  0 < s < s_\epsilon, \\
s^{\frac{\Lambda}{\Lambda+1}}I(s) & \geq
I(s_{\epsilon})s_{\epsilon}^{\frac{1}{a+\epsilon}}\times s^{\frac{\Lambda}{\Lambda+1}-\frac{1}{a+\epsilon}}.
\end{align*}
Since $[\Lambda +1] / [\Lambda] > B$, one can choose $\epsilon$ such that $\frac{\Lambda +1}{\Lambda} > a + \epsilon \implies \frac{\Lambda}{\Lambda+1}-\frac{1}{a+\epsilon}<0$. Hence $
 s^{\frac{\Lambda}{\Lambda+1}}I(s) \to \infty \text{ when } s\to 0.$ 
\end{proof}
\begin{remark} \label{A-a-beta} \normalfont
If $A$ and $B$ exist then indeed  $[\Lambda +1] / [\Lambda] > A> B$ in Proposition \ref{P-2}. However, $A$ is finite only if $B$ is finite, due to
 $P(s)I(s)\in C(0,M]$.
\end{remark}
Next, we define the  equivalence of functions and almost monotone functions. 
\begin{definition} \normalfont
\begin{enumerate}
    \item 
Functions $f$ and $g$ on a set $E$ are {\it equivalent} and write
 $f (x)\asymp  g(x)$ on $E$, if there exists a constant $c\geq1$ such that $ \displaystyle {c}^{-1}  g(x) \leq f(x) \leq c \ g(x)$ for all $x\in E$. 
 \item
A function $f$ on an interval $I\subset \mathbb{R}$ is called \textit{almost decreasing (almost increasing)} if it is equivalent to a nonincreasing (nondecreasing) function of an interval.
Equivalently, $\exists$  $c>0$ such that 
$
f(t)\geq cf(s), \text{ where } t,s\in I \text{ and }t<s.
$
\end{enumerate}
\end{definition} 
 In the following proposition, we provide the sufficient conditions for the assumption \ref{A-1}. 
\begin{proposition} \label{P-3} \normalfont
Let  $\exists$ $\mu > 0$, such that $s\mapsto P(s)I^\mu(s)$ is an almost decreasing function. Then  \ref{A-1} holds.
\end{proposition}
\begin{proof} By direct computation, it follows from \eqref{F-result} that 
\begin{align}
  F(s) \asymp  \  &      \ s^{-1}I^{-\frac{1}{\Lambda}-1}(s) \label{R-1} \\
  F^{'}(s) \asymp  \  &  \  \frac{F(s)}{s}[P(s)I(s)]^{-1} \label{R-2}.
\end{align}
Using \eqref{R-1}, \eqref{R-2} and  \ref{G-def}
\begin{align}
\frac{F(s)}{ [G(s)G^{'}(s)]} & = 
    \frac{\displaystyle F(s)}{\displaystyle \left(\int_{0}^{s} \sqrt{F^{'}(t) \ dt} \right)\displaystyle  \left( \sqrt{F^{'}(s)}\right) }\\
    & \asymp \frac{\displaystyle F(s)}{\displaystyle \left( \int_{0}^{s} \sqrt{\displaystyle \frac{F(t)}{t}[P(t)I(t)]^{-1}} \ dt\right)\left( \sqrt{\displaystyle \frac{F(s)}{s}[P(s)I(s)]^{-1}}\right) }\\
     & =
    \frac{\displaystyle [I(s)]^{-\frac{1}{2\Lambda}-\frac{1}{2}} [P(s)I(s)]^{\frac{1}{2}}}{\displaystyle \int_{0}^{s} t^{-1}[I(t)]^{-\frac{1}{2\Lambda}-\frac{1}{2}} [P(t)I(t)]^{-\frac{1}{2}} \ dt}\\
     & =
    \frac{\displaystyle [I(s)]^{-\frac{1}{2\Lambda} - \frac\mu2} [P(s)I^{\mu}(s)]^{\frac{1}{2}}}{\displaystyle \int_{0}^{s} t^{-1}[I(t)]^{-\frac{1}{2\Lambda}-1} [P(t)]^{-\frac{1}{2}} \ dt} \label{after 87}\ . 
\end{align}
By the Cauchy's mean value theorem,
there exists $t\in(0,s)$ such that
\[
\frac{\displaystyle [I(s)]^{-\frac{1}{2\Lambda} - \frac\mu2}}{\displaystyle \int_{0}^{s} t^{-1}[I(t)]^{-\frac{1}{2\Lambda}-1} [P(t)]^{-\frac{1}{2}} \ dt}
= \left(\frac{1}{2\Lambda} + \frac\mu2\right)
\frac{\displaystyle [I(t)]^{-\frac{1}{2\Lambda} - 1 - \frac\mu2}[tP(t)]^{-1}}{\displaystyle t^{-1}[I(t)]^{-\frac{1}{2\Lambda}-1} [P(t)]^{-\frac{1}{2}}}
\asymp [P(t)I^{\mu}(t)]^{-\frac12}\ .
\]
Since $t\mapsto P(t)I^{\mu}(t)$ is almost decreasing, one has
\begin{equation}
 \frac{F(s)}{G(s)G^{'}(s)} \asymp 
 \left[\frac{P(s)I^{\mu}(s)}{P(t)I^{\mu}(t)}\right]^{\frac12} 
\leq C \ ; \  0<t<s . \qedhere
\end{equation} 
\end{proof}
The sufficient condition for the existence of $\mu>0$ for  almost decreasing function  $t\mapsto P(t)I^{\mu}(t)$ is given by the next remark.
\begin{remark}\label{I-B-prop}
Let $\exists$ $\tilde P\in C^1(0,\infty)$ such that
 $P(s)\asymp \tilde P(s)$ on $s \in (0,M]$, and
 \begin{equation}\label{R-4}
     \limsup_{s \to 0}   s\tilde I(s) \tilde P^{'}(s) <\infty, 
 \end{equation}
\text{ where } $ \displaystyle \tilde I(s) \triangleq  \int_s^\infty \frac{dt}{t\tilde P(t)}$. Consequently, 
 $s\mapsto s\tilde I(s) \tilde P^{'}(s)$ is bounded on 
 $(0,M]$. Let
 \[
 B=\sup\limits_{0<s<M}s\tilde I(s) \tilde P^{'}(s).
 \]
 Fix  $\mu\geq B$. Then
 the function $Q(s) = \tilde P(s)\tilde I^{\mu}(s)$ is nonincreasing since
 \[
 Q'(s) = \tilde P^{'}(s)\tilde I^{\mu}(s) - \mu s^{-1}\tilde I^{\mu-1}(s) = s^{-1}\tilde I^{\mu-1}(s)
 \left[s\tilde I(s) \tilde P^{'}(s) - \mu\right]\leq 0.
 \]
 Finally, note that $P(s) I^{\mu}(s)\asymp\tilde P(s)\tilde I^{\mu}(s)$.
\end{remark} 
\begin{section}{Auxiliary integral estimates  }\label{iterative ineq}
We will start with auxiliary generic estimates for function $u,$ and cutoff function $\theta$ w.r.t. the nonlinear functions $F,$ and $G$.   
\begin{lemma} \label{lemma-1}
Let $u$, $\nabla u$ be measurable functions, and let $\theta \in Lip_{c}(\Omega)$. Let $F, G\in C^1(0,M)\cup C[0,M]$ satisfy \ref{A-1}. Then
\begin{equation}
\nabla u\cdot \nabla\left( \theta^{2} F(u)\right) \geq 
\frac{1}{2}| \nabla (\theta G(u))|^{2} - (2C_{1}^{2}+1) G^{2}(u)\vert\nabla \theta \vert^{2}.
    \label{lemma-1-R}
\end{equation}
\end{lemma}
\begin{proof}
We compute
\begin{align}
\nabla u\cdot \nabla  \left( \theta^{2}F(u)\right)
& = \theta^{2}F^{'}(u) \vert\nabla u \vert^{2} + 2F(u)\nabla u\cdot \theta \nabla \theta \nonumber\\
& = \theta^{2}F^{'}(u)\vert\nabla u \vert ^{2} + 2\dfrac{F(u)}{G^{'}(u)}
G^{'}(u)\nabla u\cdot \theta\nabla \theta.
\end{align}
Using \ref{A-1} in assumption in right-hand side of above yields
\begin{align}
&= \theta^{2}\vert G^{'}(u)\vert^{2} \vert\nabla u\vert^2 + 2\frac{F(u)}{G^{'}(u)}  \left( G^{'}(u) \nabla u \cdot \theta  + G(u) \nabla \theta \right) \cdot \nabla \theta -2\frac{F(u)}{G^{'}(u)}G(u)|\nabla \theta|^{2} \nonumber \\
& = |\nabla(\theta G(u))-G(u)\nabla \theta|^2 
+ 2 \dfrac{F(u)}{G^{'}(u)} \nabla(\theta G(u)) \cdot \nabla \theta -2 \dfrac{F(u)}{G^{'}(u)} G(u)\vert\nabla\theta\vert^{2} \nonumber\\
&=|\nabla(\theta G(u))|^{2}
+2\left[  \dfrac{F(u)}{G^{'}(u)}- G(u) \right]\cdot \nabla \theta \cdot\nabla(\theta G(u))\nonumber - 2\left[ \dfrac{F(u)}{G^{'}(u)} - G(u) \right]|\nabla \theta |^{2}G(u). \nonumber
\\
&\geq |\nabla(\theta G(u))|^{2}
- \left \vert 2\left[  \dfrac{F(u)}{G^{'}(u)}- G(u) \right]\cdot \nabla \theta \cdot\nabla(\theta G(u)) \right \vert
-2\left[ \dfrac{F(u)}{G^{'}(u)} - G(u) \right]|\nabla \theta |^{2}G(u).
\label{before-cauchy}\end{align}
By Cauchy's Inequality,
\begin{align}
2\left[  \dfrac{F(u)}{G^{'}(u)}- G(u) \right]\nabla \theta \cdot\nabla(\theta G(u))
\leq
2\left[  \dfrac{F(u)}{G^{'}(u)}- G(u) \right]^{2} |\nabla \theta |^{2} + \dfrac{1}{2}|\nabla(\theta G(u))|^{2}. \nonumber
\end{align}
Then \eqref{before-cauchy} becomes 
\begin{align}
\nabla u\cdot \nabla\left[  \theta^{2}F(u)\right] & \geq \dfrac{1}{2}|\nabla(\theta G(u))|^{2}
- 2\left[  \dfrac{F(u)}{G^{'}(u)}- G(u) \right]^{2}\cdot |\nabla \theta |^{2}
-\left[ 2 \dfrac{F(u)}{G^{'}(u)} - G(u) \right]|\nabla \theta |^{2}G(u) \nonumber \\
& = \dfrac{1}{2}|\nabla(\theta G(u))|^{2}- \left[  2 \left[\dfrac{F(u)}{G^{'}(u)}\right]^{2} -2\dfrac{F(u)}{G^{'}(u)} G(u) + G^{2}(u) \right] |\nabla \theta|^{2}\\
 &\geq
\frac{1}{2}| \nabla (\theta G(u))|^{2} - (2C_{1}^{2}+1) G^{2}(u)|\nabla \theta|^{2}. \qedhere
\end{align}
\end{proof}
\begin{lemma}\label{lemma-2}
Assume \ref{A-2} holds. Let $u$ be a measurable function on $\Omega$ such that $u \in (0 ,M]$, $\nabla u\in L^2_{loc}(\Omega \times (0,T)) $ and $\nabla G(u) \in L^2_{loc}(\Omega \times [0,T])$. Let $\displaystyle j={2} /{[N-2]} > 0$ for Gagliardo – Nirenberg – Sobolev inequality \begin{align}
\vert\vert\psi\vert\vert_{L^{2+2j}}^{2}
\leq
S \vert\vert \nabla \psi \vert\vert_{L^{2}}^{2}.\label{Gil-Nir}
\end{align}
Let $K\Subset \Omega$, and $ \theta_{n} \in Lip_{c} (\Omega)$ be such that $ \theta_{n} = 1$ on $K$.
Then
\begin{align}
\int_{0}^{t} \int_{K} G^{2}(u) dxdt   
\leq c_1
t^{1-(1+j)k}
\left[  \sup_{0\leq \tau \leq T} \int_{\Omega} \theta_{n}^{2} H (u(\tau)) dx
+
 \int_{0}^{t}\int_{\Omega} |\nabla(\theta_{n} u)|^2 dx d\tau
    \right]^{1+jk},\label{lemma-2-R}
\end{align}
where $ \displaystyle k =[2-\lambda] / [2+2j-\lambda]$, $c_1=C_2^{1-k}S^{k(1+j)}$ with $\lambda$ and $C_2$ as in \ref{A-2}.
\end{lemma}
\begin{proof}
Note that $\displaystyle \lambda = [2 - (2+2j)k] / [1-k]$. By \eqref{before-beta}, recall that $G(s)^{\lambda}\leq {C_{2}H(s)}$. Then 
\begin{align}
G^{2} (u) \leq C_{2}^{1-k} G^{2(1+j)k}(u)H^{1-k}(u). \label{G-H-k}
\end{align}
Integrate both side of \eqref{G-H-k} over $ K \times (0,t)$, we have
\begin{align*}
\int_{0}^{t} \int_{K} G^{2}(u) dxdt  
\leq & 
\ C_2^{1-k} \int_{0}^{t}\int_{K}G^{2(1+j)k}(u)H^{1-k}(u) dxd\tau \nonumber \\
\leq & 
\ C_2^{1-k} \int_{0}^{t}\int_{K} \left(|\theta_{n}G(u)|^{2(1+j)}\right)^{k}
\left( \theta_{n}^{2}H(u)\right)^{1-k} dx d\tau\\
\leq &
\ C_2^{1-k} \int_{0}^{t} 
    \left[\int_{\Omega} |\theta_{n}G(u)|^{2(1+j)} dx\right ]^{k}
\left[ \int_{\Omega}\theta_{n}^{2}H(u) dx\right]^{1-k}   d\tau\\
\leq&
\ C_2^{1-k}S^{k(1+j)}
 \int_{0}^{t} 
    \left[\int_{\Omega} |\nabla(\theta_{n}G(u))|^{2} dx\right]^{(1+j)k}d\tau \left[\sup_{0\leq \tau \leq T}\int_{\Omega}\theta_{n}^{2}H(u) dx\right]^{1-k}, \nonumber 
\end{align*}
by \eqref{Gil-Nir}. We apply the Holder inequality for time integral and then using the estimate $ x^{v}y^{w}\leq (x+y)^{v+w} \ ; \ x,y,v,w > 0 $ to get the following
\begin{align}
\int_{0}^{t} \int_{K}  G^{2}(u) & dxdt \nonumber   \leq    C_2^{1-k}S^{k(1+j)} \left[\sup_{0\leq \tau \leq T}\int_{\Omega}\theta_{n}^{2}H(u) dx\right]^{1-k} t^{1-k(1+j)}
 \left[\int_{0}^{t} \int_{\Omega} |\nabla(\theta_{n}G(u))|^{2} dx  d\tau\right]^{(1+j)k} \nonumber \\
& \leq
 C_2^{1-k}S^{k(1+j)} t^{1-k(1+j)} \left[\sup_{0\leq \tau \leq T}\int_{\Omega}\theta_{n}^{2}H(u) dx + \int_{0}^{t} \int_{\Omega} |\nabla(\theta_{n}G(u))|^{2} dx  d\tau\right]^{1+jk}. \nonumber  \qedhere
\end{align}
\end{proof}
In the proof of the main theorem \ref{Main-T} on localisation, we used the following iterative inequality in { \cite{LAU}}.
\begin{Lad-Ur}\label{Lady-lemma} 
Let sequence $y_n$ for $n=0,1,2,...$, be nonnegative sequence satisfying the recursion inequality,
$ \displaystyle  y_{n+1}\leq c\text{ }b^n \text{ }y_n^{1+\delta} $ with some constants $ c ,\delta > 0 \text{ and } b\geq 1$. Then
\[ \displaystyle y_n \leq c^{\frac{(1+\delta)^n -1}{\delta}}\text{ } b^{\frac{(1+\delta)^n -1}{\delta^2} -\frac{n}{\delta}}\text{ }y_0^{(1+\delta)^n}.\]
In particular $ \displaystyle \text{if} \quad y_0 \leq \theta_L = c^\frac{-1}{\delta} \text{ } b^\frac{-1}{\delta^2} \text{ and } {b > 1}$, then  $y_n \leq \theta\text{ } b^{\frac{-n}{\delta}}$ 
and consequently, \[ \displaystyle y_n \rightarrow 0 \text{ when } n\rightarrow \infty.\]
\end{Lad-Ur}

\section{Localization of the solution of degenerate Einstein Equation}\label{localization}
We prove the main theorem of localization property by constructing De-Georgi's machinery to establish the corresponding iterative energy inequality.
\begin{definition}  \normalfont \label{R-teta}
$ R > 0$ and $ b > 2$. Let $R_{n}$ be a decreasing sequence with
\begin{align}
   &  R_{n} \triangleq  R_{n-1}-Rb^{-n}
          =  R\left[\frac{ b-2+ b^{-n}}{b-1} \right] ;n=1,2, \dots.
\end{align}
 Then $ R_{0} \triangleq  R$ and $\lim\limits_{n\to\infty} R_n =  R[b-2] \big/ [b-1]$.  Let  $\theta_{n}(x) \in Lip_{c}(\Omega)$ such that
\begin{align}
&\theta_{n}(x)
\triangleq
   \left[\dfrac{\left( R_{n}-\vert\vert x-x_{0}\vert\vert\right)_{+}}{R_{n}-R_{n+1}} \wedge 1
 \right]
 = 
\begin{cases} 
 0 \ ; \ x \notin B_{n}(x_{0}),\\
 1 \ ; \ x \in B_{n+1}(x_{0}),
\end{cases} 
\end{align}
\end{definition}
where $ B_{n} \triangleq B_{R_{n}}(x_{0}) \Subset \Omega^{}$. Then one can show that
$ \displaystyle
\vert\vert \nabla \theta_{n}(x) \vert \vert_{\infty} \leq  [{b^{n+1}}] / {R}.
$
\begin{theorem}\label{Main-T}
Let $u$ be a positive solution of IBVP \eqref{ibvp}. Let $\Omega'\subset\Omega$ be such that $u(x,0)=0$ for $x\in\Omega'$. Then for every ball $B_{R}(x_{0}) \Subset {\Omega}^{'}$ and every $R'\in(0,R)$
,\ there exists $ T' > 0 $ such that $u(x,t) = 0$ for $(x,t)\in B_{R'}(x_0)\times [0,T']$.
\end{theorem}
\begin{proof}
We set $
 \displaystyle [{b-2}] \big/ [{b-1}] = {R'} / {R}  
$. Multiplying inequality in IBVP \eqref{ibvp} by $\theta_{n}^{2}$ and integrate over $\Omega \times (0,t)$, we find that
\begin{align}
\int_{\Omega} \theta_{n}^{2}H(u)  \text{ } dx
+
\int_{0}^{t}\int_{\Omega} \nabla u \nabla(\theta_{n}^{2}F(u))  \text{ } dxd\tau \leq 0. \label{fsp-1}
\end{align}
Using Lemma \ref{lemma-1} in \eqref{fsp-1} we get
\begin{align}
\int_{\Omega}  \theta_{n}^{2}H(u)   \text{ } dx  +
\frac{1}{2} \int_{0}^{t}\int_{\Omega}| \nabla(\theta_{n}G(u))|^{2} \text{ } dxd\tau 
 \leq
(2C_{1}^2+1) \int_{0}^{t}\int_{B_{n}} G^{2}(u) |\nabla \theta_{n}|^2  \text{ } dxd\tau . \label{fsp-2}
\end{align}
In particular, $\nabla G(u) \in L^2_{loc}(\Omega\times [0,T])$. Using  Lemma \ref{lemma-2} in \eqref{fsp-2}, we obtain
\begin{multline}
\int_{\Omega}  \theta_{n}^{2}H(u)  \text{ } dx  +
\frac{1}{2} \int_{0}^{t}\int_{\Omega}| \nabla(\theta_{n}G(u))|^{2} \text{ } dxd\tau\\
    \leq  
     D \ [b^{2}]^{n-1}
    t^{1-k(1+j)} 
    \left[\sup_{0\leq \tau \leq t}\int_{B_{n}}\theta_{n}^{2}H(u) dx + \int_{0}^{t} \int_{B_{n}} |\nabla(\theta_{n}G(u))|^{2} dx  d\tau\right]^{1+jk},  
\end{multline}
where $ D \triangleq  [{b^{4}}(2C_{1}^{2}+1)C_2^{1-k}S^{k(1+j)}] \big / {R^{2}}.
$
As $ 0 < t \leq T' $, by taking the supremum over $t$
\begin{multline}
\sup_{0 \leq \tau \leq T' } \int_{B_{n}} \theta_{n}^{2} H(u)  \text{ } dx
+
\int_{0}^{T'}\int_{B_{n}} |\nabla(\theta_{n}G(u)) |^{2}  \text{ } dxd\tau\\
 \leq  D \  [{b^2}]^{n-1} (T')^{1-k(1+j)} \left[\sup_{0 \leq \tau \leq T' }\int_{B_{n}}\theta_{n}^{2}H(u) dx + \int_{0}^{T'} \int_{B_{n}} |\nabla(\theta_{n}G(u))|^{2} dx  d\tau\right]^{1+jk}. \label{sup-eq}
\end{multline}
Let $\beta \triangleq [{1-(1+j)k}]/{k j}$. Note that  $\beta > 0$ and
 $\beta +1-(1+j) k = \beta(1+k j)$.
Multiply both sides of \eqref{sup-eq} by $ [T']^{\beta}$, we get
\begin{align}
& [T']^{\beta} \sup_{0 \leq \tau \leq T' } \int_{\Omega} \theta_{n}^{2}(x)H(u) \ dx
+
[T']^{\beta}\int_{0}^{T}\int_{\Omega} |\nabla(\theta_{n}(x)G(u)) |^{2}  \text{ } dxd\tau \nonumber \\
&\leq 
D \cdot ({b^2})^{n-1} \left[[T']^{\beta} \sup_{0 \leq \tau \leq T' }\int_{B_{n}}\theta_{n}^{2}(x)H(u) dx +[T']^{\beta} \int_{0}^{T'} \int_{B_{n}} |\nabla(\theta_{n}(x)G(u))|^{2} dx  d\tau\right]^{1+jk}. \nonumber
\end{align}
Observe that  $ \supp{\theta_{n}} = B_{{n}}$. Then above becomes
\begin{multline}
[T']^{\beta} \sup_{0 \leq \tau \leq T' } \int_{B_n} \theta_{n}^{2}(x)H(u) \ dx
+
[T']^{\beta}\int_{0}^{T}\int_{B_n} |\nabla(\theta_{n}(x)G(u)) |^{2}  \text{ } dxd\tau \\ \leq 
D \cdot ({b^2})^{n-1} \left[[T']^{\beta} \sup_{0 \leq \tau \leq T' }\int_{B_{n-1}}\theta_{n-1}^{2}(x)H(u) dx +[T']^{\beta} \int_{0}^{T'} \int_{B_{n-1}} |\nabla(\theta_{n-1}(x)G(u))|^{2} dx  d\tau\right]^{1+jk}. 
\label{T-beta}
\end{multline}
Define
\begin{align}
Y_{n} [T'] \triangleq [T']^{\beta} \sup_{0 \leq \tau \leq T' }\int_{B_{n}}\theta_{n}^{2}(x)H(u) dx +[T']^{\beta} \int_{0}^{T'} \int_{B_{n}} |\nabla(\theta_{n}(x)G(u))|^{2} dx d\tau. \label{En-fun}
\end{align}
Then \eqref{T-beta} yields the iterative inequality
\begin{align}
    Y_{n}[T'] \leq  D \cdot ({b^2})^{n-1} Y_{n-1}^{1+k j}[T'].
\end{align}
Let $T'$ in $\eqref{En-fun}$ be such that
 $  \displaystyle Y_{0}[T']\leq D^{-\dfrac{1}{k j}} b^{-\dfrac{2}{k^2 j^2}}.$
Then by Ladyzhenskaya-Uraltceva iterative Lemma \cite{LAU},
 $ 
Y_{n}[T']\rightarrow 0$ whenever   $n \rightarrow \infty$.
\end{proof}
\begin{remark}
In fact, it is sufficient to assume that $u$ has certain bounds and positivity on $B_R(x_0)\times [0,T]$. It is enough to assume that nonnegative solution of Cauchy problem, belong to the class of bounded functions in $R^N\times [0,\infty),$ and then  
 use the maximum principle to prove that, $u$ is bounded  by initial data at  $R^N\times \{0\}$.
\end{remark}
\end{section}
\begin{section}{Models for Degeneracy} \label{Exmp-P-F}
Without loss of generality, in this section, we will assume that $u \in (0,1]$, and  illustrate some generic examples of the function $P$, for which hold all constrains on the functions $F$, $G$ and $H$, with the following summarized remark. 
\begin{remark}\label{P-test}
Let  $P$ and $I\in C^1(0,\infty)$ be as in Definition \ref{PIH}. Assume that
\begin{align}
    \limsup\limits_{s\to0} P(s)I(s) & < \infty,\\
    \limsup\limits_{s\to0} sP^{'}(s)I(s) & < \infty.
\end{align}
Then Propositions \ref{P-2}, Propositions
\ref{P-3} and remark \ref{I-B-prop} hold. Consequently, Assumption \ref{Assumps} justifies and, thus Theorem \ref{Main-T} on finite speed of propagation asserts.
\end{remark}
Next, we provide examples  of the function $P$ in remark \ref{P-test}, where $ P(0) = 0$ and $ P \in C[0,\infty)$. 

\begin{exmpl}\label{Ex-1}{$P(s)=s^{\beta}$,  where  $s\in [0,\infty)$ and $\beta > 0$}. \normalfont
\begin{align}
 I(s) & =  \int_{s}^\infty t ^{-\beta-1} d t =  \ {\beta^{-1}} s^{-\beta}.\\
 P(s)I(s) & =  {\beta^{-1}}.  \\
 sP'(s)I(s) & = 1. 
\end{align}
\end{exmpl}
\begin{exmpl}{$ P(s)= \exp \left(-\dfrac{1}{s^{\beta}}\right), s \in [0,1), \beta > 0$}.\normalfont
\begin{align}
 I(s) = \  &  \int_{s}^{1} t^{-1}\exp \left(\dfrac{1}{t^{\beta}}\right) dt .\\
P(s)I(s) 
=  \ &  {\exp\left(-\dfrac{1}{s^{\beta}}\right)}\left[\int_{s}^{1} t ^{-1}\exp  \left(\dfrac{1}{t^{\beta}}\right) dt \right].
\end{align}
By L'Hôpital's rule 
\begin{align} \displaystyle
\lim_{s \to 0} P(s)I(s) \equiv  \ \lim_{s \to 0} s^{\beta} \equiv  \  0 \ ,
\end{align}  verifies that there exists $ \ \Lambda $ such that $ P(s)I(s) < \displaystyle \frac{\Lambda+1}{\Lambda} \ , \ s \in [0,1)$. 
Then
\begin{align}
    sI(s)P^{'}(s) =  \ &  \  s \left[\int_{s}^{1} t ^{-1}\exp \left(\dfrac{1}{t^{\beta}}\right) dt \right] \exp \left(-\frac{1}{s^{\beta}}\right)s^{\beta-1},\\
 \lim_{s \to 0} sI(s)P^{'}(s)   =  & \lim_{s \to 0}
    \frac{ \displaystyle \left[\int_{s}^{1}t^{-1} \exp \left( \frac{1}{t^{\beta}}\right) dt \right]s^{-\beta}}{ \beta \displaystyle \exp \left( \frac{1}{s^{\beta}}\right)}.
\end{align}
By L'Hôpital's rule
\begin{align}
 \lim_{s \to 0} sI(s)P^{'}(s)   
    \equiv & \ 1 +  \lim_{s \to 0}  \frac{ \left[\displaystyle \int_{s}^{1}t^{-1} \exp \left( \frac{1}{t^{\beta}}\right) dt\right]}{ \displaystyle \exp \left( \frac{1}{s^{\beta}}\right)}  \\
  \equiv & \ 1 +  \lim_{s \to 0}  s^{\beta} \\
   = & \ 1.
\end{align}
Then $ \displaystyle
\lim_{s \to 0}{F(s)}[G(s)G^{'}(s)]^{-1}
\equiv  \  1 $, which verifies  the existence of $\  C_{1}  >  0$ in  Assumption \ref{A-1}.
\end{exmpl}
\begin{exmpl}{$ \displaystyle P(s)= \exp \left(-\int_{s}^{1} \frac{\zeta(\tau)}{\tau} d\tau\right), s \in (0,1]$} \normalfont and $ 0 < k_{1} < \zeta(s) < k_{2}$.  \begin{align}
    \displaystyle I(s) =\int_{s}^{1} t^{-1}\exp \left(\int_{t}^{1} \frac{\zeta(\tau)}{\tau} d\tau\right) \ dt. 
    \end{align}
Then
\begin{align}
 \displaystyle P(s)I(s) 
     = & \  \frac{ \displaystyle\left[\int_{s}^{1} t^{-1}\exp \left(\int_{t}^{1} \frac{\zeta(\tau)}{\tau} d\tau\right) \ dt  \right]}
     {\displaystyle\left[ \exp \left(\int_{s}^{1} \frac{\zeta(\tau)}{\tau} d\tau\right)\right]}. 
\end{align}
By L'Hôpital's rule
\begin{align}
   \lim_{s \to 0} P(s)I(s) \equiv & \  \lim_{s \to 0} \frac{1}{\zeta(s)}   \equiv  1,  \label{ex-3}
      \end{align}
which verifies that there exists $  \ \Lambda $ such that $ P(s)I(s) < \displaystyle \frac{\Lambda+1}{\Lambda} \  \forall \ s \in [0,1)$. Note that $\zeta(s) \equiv 1$. 
Then
\begin{align}
s I(s)P'(s) = & \ {\displaystyle\left[\int_{s}^{1} t^{-1}\exp \left(\int_{t}^{1} \frac{\zeta(\tau)}{\tau} d\tau\right) \ dt \right] }{\displaystyle \exp \left(-\int_{t}^{1} \frac{\zeta(\tau)}{\tau} d\tau\right)\zeta(s)},\\
\lim_{s \to 0} s I(s)P^{'}(s) \equiv  \ & \lim_{s \to 0}  \frac{\displaystyle\left[\int_{s}^{1} t^{-1}\exp \left(\int_{s}^{1} \frac{\zeta(\tau)}{\tau} d\tau\right) \ dt \right]}{\displaystyle\exp \left(\int_{t}^{1} \frac{\zeta(\tau)}{\tau} d\tau\right)}.
\end{align}
We apply L'Hôpital's rule to get
\begin{align}
\lim_{s \to 0} s I(s)P'(s) \equiv \ \frac{1}{\zeta(s)}  \equiv \ 1.  \end{align}
Hence  $\displaystyle
\lim_{s \to 0}{F(s)}[G(s)G'(s)]^{-1}
\equiv  \  1 $, which verifies  that there exists $  \  C_{1}  >  0$ in  Assumption \ref{A-1}. 
\end{exmpl}
\begin{exmpl}\label{Ex-4}{\normalfont Let $\zeta$  be such that $\displaystyle 0 < \ k_{3} \leq \frac{\zeta}{\zeta_{0}} \leq k_{4}$  for some $\zeta_{0}^{'} \leq 0$  and $\displaystyle \sup_{0<s<1}{s|\zeta_{0}^{'}|}{\zeta_{0}^{-1}} =c_{0} < \infty$. Here  $ \displaystyle \lim_{\tau \to 0} \zeta(\tau) = \infty$. Provided $ \displaystyle P(s)= \exp \left(-\int_{s}^{1} \frac{\zeta(\tau)}{\tau} d\tau\right), \ s \in (0,1] $}.\normalfont  \vspace{0.2 cm}\\
Similarly in \eqref{ex-3} we get
\begin{align}
\displaystyle  \lim_{s \to 0} P(s)I(s) \equiv \lim_{s\mapsto 0} \frac{1}{\zeta(s)}  = 0,
\end{align}
verifies the  existence of $ \ \Lambda $ such that $ P(s)I(s) < \displaystyle \frac{\Lambda+1}{\Lambda} \  \forall \ s \in [0,1)$. Note that $\zeta(s) \equiv \zeta_{0}(s) $.
Then
\begin{align}
sI(s)P'(s) =  & s \left[\int_{s}^{1} t^{-1}\exp \left(\int_{t}^{1} \frac{\zeta(\tau)}{\tau} d\tau\right) \ dt\right] \exp \left(-\int_{t}^{1} \frac{\zeta(\tau)}{\tau} d\tau\right) \zeta(s).
\end{align}
\begin{align}
\lim_{s \to 0} sI(s)P^{'}(s) = &  \lim_{s \to 0} \frac{\displaystyle \left[\int_{s}^{1} t^{-1}\exp \left(\int_{t}^{1} \frac{\zeta(\tau)}{\tau} d\tau\right) \ dt\right]    \zeta_{0}(s)} {\displaystyle \exp \left(\int_{s}^{1} \frac{\zeta(\tau)}{\tau} d\tau\right)}.
\end{align}
By L'Hôpital's rule
\begin{align}
\lim_{s \to 0} sI(s)P^{'}(s) 
\equiv   & \ 1 + \lim_{s \to 0} \frac{\left[\displaystyle \int_{s}^{1} t^{-1}\exp \left(\int_{t}^{1} \frac{\zeta(\tau)}{\tau} d\tau\right) \ dt \right] \frac{ \displaystyle s \vert \zeta_{0}^{'}(s)\vert}{ \displaystyle \zeta(s)}}{\displaystyle \exp \left(\int_{s}^{1} \frac{\zeta(\tau)}{\tau} d\tau\right) \ dt} \nonumber \\
\equiv   & \ 1 +  \lim_{s \to 0}  \frac{C_{0}}{\zeta(s)} 
=  \ 1.
\end{align}
Hence $ \displaystyle
\lim_{s \to 0}{F(s)}[G(s)G'(s)]^{-1}
\equiv  \  1 $, which verifies existence of $C_{1}  >  0$ in  Assumption \ref{A-1}. 
\end{exmpl}
Note that one can  consider $P$ functions  in a way $P\asymp \tilde{P}$ as in remark \ref{I-B-prop}, for more general examples.
\end{section}
\begin{section}{Auxiliary properties of Functional spaces, and a priory estimates for the solution }
\begin{theorem} \label{dual-inclusion}
Let $X, Y$ be Banach spaces such that $X \subset Y$. Then $Y^{*} \subset X^{*}$, where $^*$ denotes the dual space.
\end{theorem}
\begin{proof}
Let $f \in Y^{*}$. Then $f(x)$ is well defined for all $x\in X \subset Y$. Moreover, 
\begin{align}
\sup_{\vert\vert x \vert \vert_{X}=1} | f(x) | 
&\leq
 \sup_{\vert\vert x \vert \vert_{X}=1} \vert\vert f \vert \vert_{Y^{*}} \vert\vert x \vert \vert_{Y}
\leq \vert\vert f \vert \vert_{Y^{*}} \implies \vert\vert f \vert \vert_{X^{*}} \leq \vert\vert f \vert \vert_{Y^{*}} \qedhere
\end{align}
\end{proof}
\begin{theorem}
Let $X,Y$ be Banach spaces and $Z$ be a topological vector space: $X,Y \subset Z$. Let $\displaystyle  X+Y \triangleq  \big \{z=x+y  \ | \ x\in X, y \in Y \big \}$, where
$\displaystyle
\vert \vert z\vert \vert_{X+Y} =  \inf \big [\vert \vert x\vert \vert_{X} + \vert \vert y\vert \vert_{Y}\big]$. Then
\begin{enumerate} [label={\normalfont (B\arabic*)}]
\item  $X+Y$ is a Banach space.
\item  $X,Y \subset [X+Y]$.
\item  $[X+Y]^{*} = X^{*} \cap Y^{*}$
\end{enumerate}
\end{theorem}
\begin{proof}
    Let $z_{n} \in X+Y$ be such that 
$ \displaystyle
\sum_{n} \vert\vert z_{n}\vert\vert_{X+Y} < \infty
$. L et $x_{n}\in X$, $y_{n}\in Y$ be such that $  x_{n}+y_{n} =  z_{n}$ and, $ 
\vert\vert x_{n}\vert\vert_{X} + \vert\vert y_{n}\vert\vert_{Y} <  \vert\vert z_{n}\vert\vert_{X+Y} + 2^{-n}, \text{ where } n \in \mathbb{N}$.
Then $ \displaystyle \sum_{n} \vert\vert x_{n} \vert\vert_{X} + \sum_{n} \vert\vert y_{n} \vert\vert_{Y} < \sum_{n} \vert\vert z_{n} \vert\vert_{X+Y}  + 2^{-n}$.
Since $X$ and $Y$ are Banach spaces, there exist $x\in X$ and $y \in Y$ such that $  \displaystyle    x = \sum_{n} x_{n}$ and $ y = \sum_{n} y_{n}$. Thus $\displaystyle \sum_{n} z_{n} =  x+y.$ \qedhere
\end{proof}
\begin{proof} Let $x \in X$ then $x \in [X+Y]$. Then
    $ \vert\vert x\vert\vert_{X+Y}  \leq \vert\vert x \vert \vert_{X}$, and
  $\vert\vert y\vert\vert_{X+Y}  \leq \vert\vert y \vert \vert_{Y}. $
    \end{proof}
\begin{proof}
Since $X \subset X+Y$ and $y \subset X+Y$, it follows from Theorem \ref{dual-inclusion},
$ [X+Y]^{*} \subset X^{*} \  \text{ and }  \ [X+Y]^{*} \subset Y^{*}$.
Hence $ [X+Y]^{*} \subset [ X^{*} \cap Y^{*}]$. Next we establish the converse. Let
$f \in [ X^{*} \cap Y^{*}]$, $z\in [X+Y]$. Let $x_{n}\in X$, $y_{n}\in Y$ be such that $ z =  x_{n} + y_{n} $ and
\begin{align}
\sum_{n} \vert\vert x_{n} \vert\vert_{X} + \sum_{n} \vert\vert y_{n} \vert\vert_{Y} < \sum_{n} \vert\vert z_{n} \vert\vert_{X+Y}  + \frac{1}{n}, \text { for } \in \mathbb{N}. \label{sum-x-y-z}\end{align}
Then $ f(z) =  f(x) + f(y) $ is well-defined. Then
\begin{align}
\vert f(z) \vert
& \leq 
    \vert\vert f \vert\vert_{X^{*}}  \vert\vert x_{n} \vert\vert_{X} +  \vert\vert f \vert\vert_{Y^{*}} \vert\vert y_{n} \vert\vert_{Y} \\
& \leq 
    \max \big[\vert\vert f \vert\vert_{X^{*}},\vert\vert f \vert\vert_{Y^{*}} \big] \ [ \ \vert\vert x_{n} \vert\vert_{X} +   \vert\vert y_{n} \vert\vert_{Y} \ ]\\
 & \leq 
     \vert\vert f \vert\vert_{X^{*} \cap Y^{*}}\left[ \vert\vert z \vert\vert_{X+Y} + \frac{1}{n}\right] \\
& \leq   \vert\vert f \vert\vert_{X^{*} \cap Y^{*}} \ \vert\vert z \vert\vert_{X+Y},
\end{align}
when $n \rightarrow \infty $.
Therefore $f \in [X+Y]^{*}$  \text{ and } 
$  \vert\vert f \vert\vert_{[X+Y]^{*}} \leq \vert\vert f \vert\vert_{X^{*} \cap Y^{*}}.$
\end{proof}
\begin{theorem}\label{compact} {\normalfont (see \cite{cpmct-SIM})}. \ Let $B_0, B_1$ and $B_2$ be Banach Spaces such that  $B_{0} \Subset B_{1} \subset B_{2}$. Let $\mathbb{F} \subset L^{P}(I,B_{0})$ ; $I \subset \mathbb{R}$. If \
$ \displaystyle
 \vert\vert f \vert\vert_{L^{P}(I,B_{0})} < \infty$ and \ 
$ \displaystyle  \vert\vert f^{'} \vert\vert_{L^{1}(I,B_{2})} < \infty
$ for $ {f\in \mathbb{F}}$,
then $ \mathbb{F} $ is compact in $L^{P}(I,B_{1})$.
\end{theorem}

 \section{weakly approximated solution of degenerate Einstein equation. }\label{WAPS}
In this section we will prove roundness of the regularised  solution $u_{\epsilon}(x,t)$ in the space defined by LHS in \eqref{uni-est}, and compactness in $L^{q}_{loc} (\Omega\times I)$ (see Corollary \ref{proposition on compactness}.) 
\begin{theorem} \label{grad-u-bound}
Let $u^{\epsilon}(x,t)$ ; $0 < \epsilon \leq u^{\epsilon} \leq K < \infty$, be a classical solution of the problem
\begin{align}\label{IBVP-G}
\text{IBVP-G} = \begin{cases} 
 {H}_t(u^{\epsilon}) -   [F(u^{\epsilon})+\epsilon]\Delta u^{\epsilon}   \  = 0 \  & \text{ in }  \Omega\times (0,T], \\
\hspace{3.1 cm} u^{\epsilon}(x,0)   \  = \epsilon + g(x)   &\mbox{ in }  \Omega,   \\
\hspace{3.1 cm} \ u^{\epsilon}(x,t)  =\hspace{0.08 cm} {\epsilon} \psi(x)      & \mbox{ on } \ {\partial \Omega \times (0,T]},
\end{cases}
\end{align}
where $g(x) \geq 0$, $g(x) \in W^{1,2}(\Omega)$ and  $\psi(x) \geq 0$. Let 
$ \displaystyle \tilde{H}(u^{\epsilon})=\int_{\epsilon}^{u^{\epsilon}} \sqrt{{h(s)}/[{F(s)}+\epsilon}] \ ds$. 
Then for any $0 < \tau \leq T $,
\begin{align}\label{energy-id}
\int_{0}^{\tau} \int_{\Omega}[\tilde{H}_t(u^{\epsilon})]^{2} dxdt 
+ \int_\Omega|\nabla u^{\epsilon}(x,\tau)|^2 dx
= \int_{\Omega} |\nabla g(x)|^2 dx.
\end{align}
%
Furthermore,
\begin{multline}
 \int_0^{T} \int_{\Omega}[\tilde{H}_t(u^{\epsilon})]^2 \ dxdt 
 +
 \int_0^{T}  \int_\Omega |\nabla u^{\epsilon}(x,t)|^2\ dxdt + 
 \int_{0}^{T} \int_{\Omega} \vert u^{\epsilon}(x,t) \vert^{2} \ dxdt \\
\leq 
C(\Omega,T) \left[\int_\Omega |\nabla g(x)|^2 \ dx 
+
K^{2} \int_{\partial\Omega} \vert \psi(x) \vert^{2} \ ds \right]. \label{uni-est}
 \end{multline}
\end{theorem}
\begin{proof}
By definition of $H$ in  \eqref{H-F-def}, we get  
\begin{align}
    &  H_t(u^{\epsilon}) =   h(u^{\epsilon}) u^{\epsilon}_t, \quad   
     [\tilde{H}_t(u^{\epsilon})]^2 = \frac{h(u^{\epsilon})}{F(u^{\epsilon}) +  \epsilon} \cdot [u^{\epsilon}_t]^2, \\
 \text{therefore }  \ &  H_t(u^{\epsilon}) -   (F(u^{\epsilon})+\epsilon)\Delta u^{\epsilon} = 0 \equiv
u^{\epsilon}_{t} \cdot {h(u^{\epsilon})}/ {[F(u^{\epsilon})+ u^{\epsilon}}]-\Delta u^{\epsilon} = 0.
\end{align}
Multiply first equation in  \eqref{IBVP-G} by ${(u_\epsilon)}_{t}$ we get
\begin{align}
\text{IBVP-H } =\begin{cases} \displaystyle
 [\tilde{H}_{t}(u^{\epsilon})]^{2}-\Delta u^{\epsilon} u^{\epsilon}_{t}  =  0,   &\ \text{ in } \Omega \times (0,T],\label{h-eq} \\
u^{\epsilon}(x,0) =  \epsilon + g(x) &\  \text{ in } \Omega,\\
u^{\epsilon}(x,t) = \epsilon \cdot \psi(x)  &\  \text{ on } \partial \Omega \times (0,T].
\end{cases}
\end{align}
Integrate over $\Omega\times (0,\zeta_{1})$ for $0 < \zeta_{1} \leq T$,
\begin{align}
\int_0^{\zeta_{1}} \int_{\Omega}[\tilde{H}_t(u^{\epsilon})]^2 \ dxdt = &
\ \int_0^{\zeta_{1}} \int_{\Omega} \Delta u^{\epsilon}u^{\epsilon}_{t} \ dxdt\\
     = & \  - \int_0^{\zeta_{1}}  \int_{\Omega} \nabla u^{\epsilon}_{t} \nabla u^{\epsilon} \ dx dt  =  -\frac{1}{2}  \int_0^{\zeta_{1}} \int_{\Omega} \left({\vert \nabla u^{\epsilon} \vert^{2}}\right)_{t} \ dx dt.
\end{align}
Integrate the right-hand side over time, we get
\begin{equation}
2\int_0^{\zeta_{1}} \int_{\Omega}[\tilde{H}_t(u^{\epsilon})]^2 \ dxdt 
+
\int_\Omega|\nabla u^{\epsilon}(x,\zeta_{1})|^2dx =
 \int_\Omega |\nabla g(x)|^2dx \ ; \ 0 < \zeta_{1} \leq T. \label{result-1}
\end{equation}
By \eqref{result-1}, we write following two estimates\ :
\begin{align}
   2 \int_0^{\zeta_{1}} \int_{\Omega}[\tilde{H}_t(u^{\epsilon})]^2 \ dxdt 
    \leq &  \  \int_{\Omega}  |\nabla g(x)|^2 \ dx \label{est-2},\\
  \int_{0}^{\zeta_{2}}  \int_\Omega|\nabla u^{\epsilon}(x,\zeta_{1})|^2  \ dx d\zeta_1 
  \leq & \
  \zeta_{2}  \int_{\Omega}  |\nabla g(x)|^2 \ dx  \ ; \ 0 < \zeta_{2} \leq T \label{est-3}.
\end{align}
By adding  \eqref{est-2} and \eqref{est-3} we get
\begin{multline}
2 \int_0^{\zeta_{1}} \int_{\Omega}[\tilde{H}_t(u^{\epsilon})]^2 \ dxdt 
 +
 \int_0^{\zeta_{2}} \int_\Omega |\nabla u^{\epsilon}(x,\zeta_1)|^2\ dxd \zeta_1
\leq 
(1+{\zeta_{2}})  \int_\Omega |\nabla g(x)|^2 \ dx .  \label{before-frid}
\end{multline}
By Friedrich's inequality \cite{Mazja}, for $u \in W^{1,2}(\Omega)$, one can write
\begin{align}
C_{F}^{-1}   \int_{0}^{\zeta_{2}} \int_{\Omega} [u^{\epsilon}(x,\zeta_1)]^{2} \ dxd\zeta_1 - \zeta_{2} \cdot \epsilon^{2} \int_{\partial\Omega} \vert \psi \vert^{2} \ ds 
  \leq \int_{0}^{\zeta_{2}} \int_{\Omega} \vert \nabla u^{\epsilon}(x,\zeta_1) \vert^{2} \ dxdt.
   \label{F-D-in}
\end{align}
Let $ \epsilon_0$ ; $\frac{1}{2} < \epsilon_{0} < 1 $ be fixed. Then we use \eqref{F-D-in} in \eqref{before-frid} yields
\begin{multline}
2 \int_0^{\zeta_{1}} \int_{\Omega}[\tilde{H}_t(u^{\epsilon})]^2 \ dxdt 
  \ + \
\epsilon_{0}  \int_0^{\zeta_{2}} \int_\Omega |\nabla u^{\epsilon}(x,\zeta_{1})|^2\ dxd\zeta_{1} \\
 \ +  \ {(1-\epsilon_{0})} C_{F}^{-1}
 \int_{0}^{\zeta_{2}} \int_{\Omega} \vert u^{\epsilon}(x,\zeta_1) \vert^{2} \ dxd\zeta_1 \nonumber \\
\leq 
(1+{\zeta_{2}})  \int_\Omega |\nabla g(x)|^2 \ dx  + (1-\epsilon_{0}){\zeta_{2}} \cdot \epsilon^{2} \int_{\partial\Omega} \vert \psi(x) \vert^{2} \ ds,
\end{multline}
where $ \zeta_{2} \in (0,T]$. Setting $\zeta_{2}=T$,
\begin{multline}
 \int_0^{T} \int_{\Omega}[\tilde{H}_t(u^{\epsilon})]^2 \ dxdt 
 +
 \int_0^{T}  \int_\Omega |\nabla u^{\epsilon}(x,t)|^2\ dxdt + 
 \int_{0}^{T} \int_{\Omega} \vert u^{\epsilon}(x,t) \vert^{2} \ dxdt \\
\leq 
C(\Omega,T) \left[\int_\Omega |\nabla g(x)|^2 \ dx 
+
 \epsilon^{2} \int_{\partial\Omega} \vert \psi(x) \vert^{2} \ ds \right], \label{uni-est-2}
\end{multline}
where $ \displaystyle C(\Omega,T) \triangleq  
\max [1+T ,(1-\epsilon_{0})T)]\big / \min[2, \epsilon_{0},(1-\epsilon_{0})C_{F}^{-1}] = [1+T]\big / \min [ \epsilon_{0},(1-\epsilon_{0})C_{F}^{-1}] $. Thus we obtain \eqref{uni-est} by replacing  $\epsilon $ by $K$ on the right-hand side of \eqref{uni-est-2}.
\end{proof}
\begin{remark} \label{grad-u-L2}
Below in the Theorem \ref{grad-G-bound} we will prove that
$u^{\epsilon}$ is uniformly bounded in $L^{2}({I, W^{1,2}(\Omega)})$.
$\displaystyle \tilde{H}_t(u^{\epsilon})$ is uniformly bounded in $L^{2}(\Omega \times (0,T))$. 
\end{remark}
\begin{theorem} \label{grad-G-bound}
Let $u^{\epsilon}$ be a family of strong solutions to problem
\begin{align} 
\text{IBVP$_{\epsilon}$ } =\begin{cases} 
 [H(u^{\epsilon})]_t -   [F(u^{\epsilon})+\epsilon]\Delta u^{\epsilon}   = 0 \  & \text{ in }  \Omega\times (0,T],
\label{ibvp-ep-2} \\
\hspace{3.3 cm} u^{\epsilon}(x,0)   \  = \epsilon    &\mbox{ in } \Omega^{'} \Subset \Omega,   \\
\hspace{3.4 cm} \ u^{\epsilon}(x,t)  =\hspace{0.08 cm} {\epsilon} \psi(x)        & \mbox{ on } \ {\partial \Omega \times (0,T]},
 \end{cases} 
 \end{align}
such that $[H(u^{\epsilon})]_t, \Delta u^{\epsilon} \in L^{2}_{loc}(\Omega \times (0,T])$, satisfying the estimate $ \displaystyle
0 < \epsilon \leq u^{\epsilon} \leq K$. Then
\begin{multline}
     \int_{\Omega \times (0,T)} |\nabla (\theta G(u^{\epsilon}))|^{2} \ dxdt \\
    \leq
    C_4 \bigg[ \int_{\Omega} \theta^{2} |G(u^{\epsilon}(x,0))|^2 \ dx +
     \int_{\Omega \times (0,T)} |\nabla \theta|^{2} (G(u^{\epsilon}))^{2} \ dxdt + K_\epsilon\bigg ], 
    \label{1-star}
\end{multline} 
for every $\theta \in C_{c}^{1}(\Omega)$.
Then  $G(u^{\epsilon})$ is uniformly bounded in $L^{2}(I, W^{1,2}_{loc}(\Omega)) $ for each $\epsilon > 0.$ 
\end{theorem}
\begin{proof}
Multiply both sides of the equality in  \eqref{ibvp-ep-2}  by $\theta^{2}$, and integrate over $\Omega \times (0,T]$:
\begin{multline}
\int_{\Omega}[H[u^{\epsilon}(x,T)] - H[u^{\epsilon}(x,0)] \theta^{2} \ dx \\ 
 =   - \int_{0}^{t}\int_{\Omega} [ \nabla u^{\epsilon}] ^{2} F^{'}(u^{\epsilon}) \theta^{2} 
- [F(u^{\epsilon})+\epsilon] \nabla u^{\epsilon} \cdot \nabla (\theta^{2}) \ dxdt 
\label{E-5}
\end{multline}
Rearrange above to the following inequality:
\begin{multline}
\int_{\Omega}[H(u^{\epsilon}(x,0)] \theta^2 \ dx   
+ 
\epsilon\int_{0}^{t}\int_{\Omega}  |\nabla u_{\epsilon} \nabla (\theta^{2})| \ dxdt  
\\
 \geq    \int_{0}^{t}\int_{\Omega} [ \nabla u^{\epsilon}] ^{2} F^{'}(u^{\epsilon}) \theta^{2}  \ dxdt 
-  \int_{0}^{t}\int_{\Omega} \vert F(u^{\epsilon}) \nabla u^{\epsilon} \cdot \nabla (\theta^{2}) \vert \ dxdt.  \label{E-5}
\end{multline}
Next, by \ref{F-def}, we get \  $[ \nabla u^{\epsilon}] ^{2} F^{'}(u^{\epsilon}) \cdot \theta^{2} 
 = \ [\nabla G (u^{\epsilon})]^2 \theta^2 $. Applying Cauchy's Inequality 
\begin{align}
 [\nabla G (u^{\epsilon})]^2 \theta^2 
  \geq \ &  [ \ \vert \nabla (\theta G(u^{\epsilon}) \vert - \vert G(u^{\epsilon}) \nabla \theta \vert \ ]^{2}\\
  \geq \ & 
 (1-2 \epsilon_1) \vert \nabla (\theta G(u^{\epsilon})) \vert^{2}
-
\left(\frac{1}{2\epsilon_1}  - 1\right)\vert G(u^{\epsilon})\nabla \theta\vert^{2}.
 \label{E-7}
\end{align}
for some fixed  $ 0 < \epsilon_1 < \frac{1}{2}$, and 
\begin{align}
    \epsilon\int_{0}^{t}\int_{\Omega}  |\nabla u_{\epsilon} \nabla (\theta^{2})| \ dxdt    
     \leq
    \epsilon \left[\int_{0}^{t}\int_{\Omega} |\nabla u_{\epsilon}|^{2} \ dxdt + 2 \int_{0}^{t}\int_{\Omega} |\nabla (\theta)|^{2} \ dxdt\right] \leq K_2({\epsilon})
\end{align}
Next, using \ref{A-1}  we compute
\begin{align}
\vert F(u^{\epsilon}) \nabla u^{\epsilon} \cdot \nabla (\theta^{2}) \vert
\leq \ &
C_1 \vert  G ^{'}(u^{\epsilon}) G (u^{\epsilon}) \ \nabla u^{\epsilon} \cdot \nabla (\theta^{2}) \vert
 = \ 
 \frac{C_1}{2}\vert \nabla ( G^2 (u^{\epsilon}) \cdot \nabla (\theta^{2})  \vert .  \label{E-8}
\end{align}
Combining \eqref{E-7} and \eqref{E-8} in \eqref{E-5} becomes 
\begin{multline}
\int_{\Omega} H[u^{\epsilon}(x,0)] \cdot \theta^{2} \  dx 
+
 K_2({\epsilon})   
+
\left(\frac{1}{2\epsilon_1}-1\right) \int_{0}^{t}\int_{\Omega}  \vert G(u^{\epsilon})\nabla \theta\vert^{2} \ dxdt 
\\  \geq 
(1-2 \epsilon_1)  \int_{0}^{t}\int_{\Omega}  \vert \nabla (\theta G(u^{\epsilon})) \vert^{2} \ dxdt 
 -  \frac{C_1}{2}  \int_{0}^{t}\int_{\Omega}\vert \nabla ( G^2 (u^{\epsilon})) \cdot \nabla (\theta^{2})  \vert \ dxdt. \label{E-8.5}
\end{multline}
Once again applying Cauchy's Inequality,  We compute the following
\begin{align}\label{ineq-1}
  (1-2 \epsilon_1)  \vert \nabla  &\ (\theta G(u^{\epsilon})) \vert^{2} 
 -  \frac{C_1}{2}  \vert \nabla ( G^2 (u^{\epsilon})) \cdot \nabla (\theta^{2})  \vert \nonumber \\
 & = (1-2 \epsilon_1)\left[    
 |\theta \nabla G(u^\epsilon)|^{2} + |G(u^\epsilon) \nabla \theta |^{2}  
 \right] 
 -
  2(1-2 \epsilon_1 + C_1) | G \nabla G(u^{\epsilon}) \theta  \nabla \theta|\nonumber\\
  & \geq
  \big[1-2 \epsilon_1 - 2\epsilon_2(1-2 \epsilon_1 + C_1) \big]|\theta \nabla G(u^\epsilon)|^{2} 
  -
   \bigg[\frac{1}{2{\epsilon_2}}(1 + C_1)+2 \epsilon_1 \bigg] |G(u^\epsilon) \nabla \theta |^{2}, 
\end{align}
where $ [1-2\epsilon_1] \big / [2(1-2\epsilon_1+C_1] > \epsilon_2$. 
Using \eqref{ineq-1} in \eqref{E-8.5}, we obtain
\begin{multline}
\int_{\Omega} H[u^{\epsilon}(x,0)] \cdot \theta^{2} \  dx 
+
K_2({\epsilon})
+
 \bigg[\frac{1}{2\epsilon_1} + \frac{1}{2{\epsilon_2}}(1 + C_1)+2 \epsilon_1 \bigg]\int_{0}^{t}\int_{\Omega}  \vert G(u^{\epsilon})\nabla \theta\vert^{2} \ dxdt 
\\  \geq 
 \int_{0}^{t}\int_{\Omega} \bigg[1-2 \epsilon_1 - 2\epsilon_2(1-2 \epsilon_1 + C_1) \bigg]|\theta \nabla G(u^\epsilon)|^{2} 
\end{multline}
Note that $H(u^{\epsilon})$ is bounded for $0 <u^\epsilon < K$ . By selecting 
\begin{align}
C_4 = {\max \bigg[1,  \bigg[\frac{1}{2\epsilon_1} + \frac{1}{2{\epsilon_2}}(1 + C_1)+2 \epsilon_1 \bigg]  \bigg]} \bigg/ {\bigg[1-2 \epsilon_1 - 2\epsilon_2(1-2 \epsilon_1 + C_1) \bigg]},
\end{align} gives  \eqref{1-star}.
\end{proof}
\begin{proposition}\label{H-bound-1}
\normalfont
Assume all conditions in Proposition \ref{P-1} and  Proposition \ref{P-2} hold. If \[ \displaystyle
\liminf_{s\rightarrow 0} {P(s)\left[\int_{s}^{M} \frac{1}{\sigma P(\sigma)} \ d\sigma \right]^{\frac1\Lambda}} > 0 \ , \]
then 
$ \displaystyle \sup\limits_{\epsilon>0} \left[\int_{0}^{t}\int_{\Omega}\theta|\nabla H(u^{\epsilon})|^2 \ dxdt \right] \leq C_\theta<\infty$, for every $\theta\in C^1_c(\Omega)$. 
\end{proposition}
\begin{proof}
Note that
$ \displaystyle
    |\nabla H(u^{\epsilon})| 
    \leq
    \left\vert \frac{H'(u^{\epsilon})}{G'(u^{\epsilon})} \right \vert |\nabla G(u^{\epsilon})|
$. By Theorem \ref{grad-G-bound}, it suffices to prove that $\displaystyle \sup_{0<s<K} \left\vert \frac{H'(s)}{G'(s)} \right\vert <\infty.
$ For this end, it is enough to verify
$ \displaystyle    \limsup_{s \rightarrow 0} \left\vert \frac{H'(s)}{G'(s)} \right\vert < \infty
$. By remark \ref{new-def}
\begin{align}
F(s)     = &  \Lambda^{-\frac{1}{\Lambda}-1}s^{-1}(I(s))^{-\frac{1}{\Lambda}-1},\\
F^{'}(s) = & B_{1} s^{-2} {(I(s))^{-\frac{1}{\Lambda}-2}}({P(s)})^{-1}\left(\frac{1+\Lambda}{\Lambda} -P(s)I(s)\right) \ ; \ B_{1}\triangleq   \Lambda^{-\frac{1}{\Lambda}-1}.
\end{align}
By \ref{F-def} in Definition \ref{G-def}
\begin{align}
    {G^{'}(s)} = & \sqrt{B_{1}} s^{-1} {(I(s))^{-\frac{1}{2\Lambda}-1}}({P(s)})^{-\frac{1}{2}}\sqrt{\frac{1+\Lambda}{\Lambda} - P(s)I(s)}  \label{F-deri-I}.
\end{align}
We write $H(s)$ in term of $I(s)$ in \eqref{H-choice}
\begin{align}
    H(s)    = & (\Lambda I(s))^{-\frac{1}{\Lambda}},\\
    H^{'}(s)= & B_{1} {(I(s))^{-\frac{1}{\Lambda}-1}}s^{-1} (P(s))^{-1}. \label{H-deri-I}
\end{align}
Applying  \eqref{H-deri-I} and \eqref{F-deri-I} we get
\begin{align}
    \left\vert \frac{H'(s)}{G'(s)} \right\vert 
    =\frac{1}{\sqrt{B_{1}}} \frac{1}{\sqrt{(I(s))^{\frac{1}{\Lambda}}P(s)}}\frac{1}{\sqrt{\frac{1+\Lambda}{\Lambda}-I(s)P(s)}} \label{ratio-1}.
\end{align}
By Proposition \ref{P-2} we have 
$ \displaystyle \frac{1+\Lambda}{\Lambda} > I(s)P(s) > 0 \ ; \ s\in [0,M) $. Thus, by \eqref{ratio-1}, it follows that  $ \displaystyle
    \limsup_{s\rightarrow 0} \left\vert \frac{H'(s)}{G'(s)} \right\vert 
    < \infty$, 
whenever $ \displaystyle
\lim_{s\rightarrow 0} \inf{P(s) (I(s))^{\frac1\Lambda}} > 0$ \ . 
\end{proof}

\begin{theorem}
Let $u^{\epsilon}$  satisfies Theorem \ref{grad-G-bound}.  Then   $u^{\epsilon}$ holds the following uniform estimates
\begin{enumerate} [label={\normalfont(H\arabic*)}]
  \item 
$ \displaystyle
\int_{\Omega' \times I}  \vert   \Psi {H}_t(u^{\epsilon})  \vert  \ dxdt   
\leq C (\Omega',K)\left[\vert\vert \Psi   \vert\vert_{L^{\infty}(\Omega\times I)}
+
\vert\vert   \nabla \Psi   \vert\vert_{L^2(\Omega\times I)}
\right], \label{L1-bound}
$
\item $  \displaystyle
\int_{\Omega' \times I}|\nabla H(u^{\epsilon})|^2\ dxdt
\leq C(\Omega',K),
$
\end{enumerate}
 on $\Omega^{'} \times {I}$, for each $\Psi\in C^1_c(\Omega^{'} \times {I})$  and for any $\Omega'\Subset\Omega$.
\end{theorem}
\begin{proof}
Multiply both sides of the inequality in  \eqref{ibvp-ep}  by $ \Psi$, and integrate over $ \Omega'_{I} \triangleq \Omega^{'} \times {I}$:
\begin{multline}
 \int_{\Omega'_{I}} \vert \Psi {H}_t(u^{\epsilon}) \vert \ dxdt   
 \leq  
\int_{\Omega'_{I}} \left\vert \nabla u^{\epsilon} \cdot \nabla [F( u^{\epsilon})\Psi]\right\vert \ dxdt \\
 \leq  
 \int_{\Omega'_{I}} \left\vert F(u^{\epsilon})
 \nabla u^{\epsilon}\cdot\nabla \Psi \right\vert
 + 
 \vert \Psi F^{'}( u^{\epsilon}) (\nabla u^{\epsilon})^{2}
 \vert \ dxdt \label{z-1}.
\end{multline}
Using \ref{A-1} on the right-hand side of \eqref{z-1}, we write
\begin{align}
 \int_{\Omega'_{I}} \vert \Psi {H}_t(u^{\epsilon}) \vert \ dxdt    & \leq 
C_{2}\int_{\Omega'_{I}} \vert G(u^{\epsilon}) G^{'}(u^{\epsilon})
   \nabla u^{\epsilon}\cdot\nabla \Psi \vert
   \ dxdt
   + 
\int_{\Omega'_{I}} \vert \Psi (G^{'}( u^{\epsilon}))^2 (\nabla u^{\epsilon})^{2} \vert \ dxdt  \nonumber
\\
& \leq 
   \frac{C_{2}}{2}\int_{\Omega'_{I}} \vert \nabla G^2(u^{\epsilon}) \vert\vert \nabla \Psi \vert  \ dxdt
   + 
   \vert\vert \Psi \vert \vert_{L^{\infty}(\Omega^{'}_{I})}\int_{\Omega'_{I}}\vert \nabla G(u^{\epsilon})\vert^2
   \ dxdt. \label{z-2}
\end{align}
Rearrange the right-hand side of \eqref{z-2}, we get
\begin{align}
    \leq  
\max \left [\frac{C_{2}}{2},1\right]\left[\vert \vert \nabla \Psi \vert \vert_{L^{2}(\Omega^{'}_{I})} + \vert\vert \Psi \vert \vert_{L^{\infty}(\Omega^{'}_{I})} \right] \cdot
\left[\int_{\Omega'_{I}}  \vert \nabla G^2(u^{\epsilon})\vert   \ dxdt 
   + 
 \int_{\Omega'_{I}} \vert \nabla G(u^{\epsilon})\vert^2
   \ dxdt \right]. \nonumber
\end{align}
Note that $0 < \epsilon \leq u^{\epsilon} \leq K $, and $G(u^{\epsilon})$ is uniformly bounded in $L^{2}(I, W^{1,2}_{loc}(\Omega)) $ by Theorem \ref{grad-G-bound}. Thus we obtain {\ref{L1-bound}}. \qedhere
\end{proof}
From above theorem follows
\begin{corollary}\label{proposition on compactness}
If we obtain an estimate analogous to \eqref{1-star}, replacing $\nabla G(u^{\epsilon})$ with $\nabla H(u^{\epsilon})$,  then we can apply Theorem \eqref{compact} to conclude compactness of $\{ H(u^{\epsilon})\}$ in $L^{2} (I,L^{q}_{loc}(\Omega))$ with $ q <\frac{2N}{N-2} $. Since $H:\mathbb{R}\to \mathbb{R}$ is a homeomorphism, and $u_\epsilon$ is uniformly bounded, it follows that  
$\{u^{\epsilon}\}$ is compact in $L^{q}_{loc} (\Omega\times I)$, $q\geq1$.
\end{corollary}
\end{section}

\bibliography{NEW-5}

\begin{thebibliography}{10}

\bibitem{Aronson-1}
D.~G. Aronson.
\newblock {\em The porous medium equation. In: Fasano A., Primicerio M. (eds)
  Nonlinear Diffusion Problems}, volume 1224 of {\em Lecture Notes in
  Mathematics}.
\newblock Springer, Berlin,Heidelberg, 1986.
\newblock \url{https://doi.org/10.1007/BFb0072687}.

\bibitem{Barenblatt-96}
G.~I. Barenblatt.
\newblock {\em Scaling, Self-similarity, and Intermediate Asymptotics}.
\newblock Cambridge University Press, 1996.
\newblock \url{10.1017/CBO9781107050242}.

\bibitem{CIL92}
M.~G. Crandall, H.~Ishii, and P.-L. Lions.
\newblock User's guide to viscosity solutions of second order partial
  differential equations.
\newblock {\em Bull. Amer. Math. Soc.}, 27:1--67, 1992.
\newblock \url{10.1090/S0273-0979-1992-00266-5}.

\bibitem{Einstein56}
A.~Einstein.
\newblock {\em Investigations on the Theory of the Brownian Movement}.
\newblock 1956.

\bibitem{Evan}
L.~C. Evans.
\newblock {\em Partial Differential Equations}, volume~19.
\newblock American Mathematical Society, 2010.

\bibitem{cpmct-SIM}
{J. Simon}.
\newblock {Compact sets in the space L$^p$(O,T ; B)}.
\newblock {\em {Annali di Matematica Pura ed Applicata}}, {146}:{65--96},
  {1986}.
\newblock \url{https://doi.org/10.1007/BF01762360}, (p.84-86).

\bibitem{kon}
K.~K{\"o}nigsberger.
\newblock {\em Analysis 2}, volume~2 of {\em Springer-Lehrbuch}.
\newblock Springer Berlin Heidelberg, 2006.
\newblock \url{https://books.google.lk/books?id=V3crjPiI-mMC}.

\bibitem{LAU}
O.~A. Lady\v{z}enskaja, V.~A. Solonnikov, and N.~N. Ural'ceva.
\newblock {\em Linear and Quasi-linear Equations of Parabolic Type}, volume~23
  of {\em Translations of Mathematical Monographs}.
\newblock American Mathematical Society, Providence, RI, 1968.

\bibitem{liberman}
G.~M. Lieberman.
\newblock {\em Second Order Parabolic Differential Equations}.
\newblock World Scientific, 1996.

\bibitem{Muskat}
M.~Muskat.
\newblock {\em The Flow of Homogeneous Fluids Through Porous Media}.
\newblock UMI Books on Demand. Ann Arbor ( Michigan ) ,UMI, 2004.

\bibitem{Krylov}
N.V.Krylov.
\newblock {\em Nonlinear Elliptic and Parabolic Equations of the Second Order}.
\newblock 1987.

\bibitem{Mazja}
V.G.Maz’ja.
\newblock {\em Sobolev Spaces}.
\newblock Springer Berlin, Heidelberg, 1985.
\newblock \url{https://doi.org/10.1007/978-3-662-09922-3}.

\bibitem{vin-krug}
W.~Vincenti and C.~Kruger.
\newblock {\em Introduction to physical gas dynamics}.
\newblock Krieger Publishing Company, 1965.

\end{thebibliography}
 \bibliographystyle{abbrv}

\end{document}